\documentclass[a4paper,fleqn]{cas-dc}

\usepackage[numbers]{natbib}
\usepackage{amsmath, amsthm, amsfonts, dsfont}
\usepackage{mathtools}

\newtheorem{theorem}{Theorem}
\newtheorem{assumption}{Assumption}
\newtheorem{proposition}{Proposition}
\newtheorem{lemma}{Lemma}

\newtheorem{corollary}{Corollary}

\def\tsc#1{\csdef{#1}{\textsc{\lowercase{#1}}\xspace}}
\tsc{WGM}
\tsc{QE}
\tsc{EP}
\tsc{PMS}
\tsc{BEC}
\tsc{DE}

\newcommand*{\rttensortwo}[1]{\bar{\bar{#1}}}
\newcommand{\ep}{\epsilon}
\newcommand{\halfspace}{\hspace{0.2mm}}

\begin{document}
\let\WriteBookmarks\relax
\def\floatpagepagefraction{1}
\def\textpagefraction{.001}

\title [mode = title]{Heavy Traffic Diffusion Limit for a Closed Queueing Network with Single-Server and Infinite-Server Stations}

\author[1]{{Amir A. Alwan}}

\author[2]{{Bar\i\c{s} Ata}}

\affiliation[1]{organization={Lubar College of Business, University of Wisconsin--Milwaukee}
                }

\affiliation[2]{organization={Booth School of Business, The University of Chicago}
                }


\begin{abstract}
This paper studies the limiting behavior of a closed queueing network with multiple single-server and infinite-server stations. 
Under a heavy traffic asymptotic regime---where the number of jobs and single-server service rates grow large while infinite-server rates remain fixed---we prove a 
weak convergence result for the queue length and idleness process vector, providing an approximation for the original system.\\\vspace{-0.5em}

\vspace{-.5em}
\end{abstract}


\maketitle

\section{Introduction}\label{sec:Introduction}

This paper studies the limiting behavior of a closed queueing network consisting of multiple single-server and infinite-server stations under a heavy traffic asymptotic regime. Jobs circulate perpetually through the system according to a two-level probabilistic routing structure: from single-server stations to infinite-server stations, and from infinite-server stations back to the single-server stations.

The asymptotic regime is one in which both the number of jobs in the system and the service rates at the single-server stations grow large, while the service rates at the infinite-server stations remain fixed. In this regime, we further assume that the system is in heavy traffic, meaning that each single-server station is critically loaded: the nominal arrival rate to each station matches its service capacity. This reflects a setting in which congestion arises at the single-server stations, while the infinite-server stations act as delay buffers. 

Our main result is a functional limit theorem for the diffusion-scaled queue length and idleness process vector, where the limit process solves a Skorokhod problem. To establish this, we prove existence, uniqueness, and continuity of the associated nonlinear regulator mapping---results which are of interest in their own right. 
Unlike several earlier works on the weak convergence of closed networks with infinite-server stations, which have employed martingale techniques, we use a continuous mapping argument. 

The model was motivated in part by ride-hailing systems such as Uber and Lyft. In this context, the single-server stations represent regions of a city where drivers await ride requests, while the infinite-server stations capture travel times between regions. The service rates at the single-server stations represent customer arrival rates in each region, while the service rates at the infinite-server stations represent the travel time completion rates. After picking up a passenger (i.e., after a service completion at a single-server station), a driver enters a travel phase to reach the destination, and upon completion (i.e., after a service completion at an infinite-server station), becomes available again in a potentially new region. After picking up a passenger, a driver enters a travel phase to reach the destination, and upon completion, becomes available again in a potentially new region. The two-level routing structure in our model naturally reflects such origin-destination dynamics, and the presence of multiple infinite-server stations allows for heterogeneity in travel times across different region pairs. 
While ride-hailing serves as a useful motivation, the model and results apply more broadly to systems in which tasks are queued for assignment or dispatch at specific locations, and then processed or delayed in parallel elsewhere. By establishing a rigorous diffusion approximation for such systems, this paper lays the theoretical groundwork for future research on their dynamic control and optimization.

Ride-hailing systems are often modeled as open, with drivers entering and leaving the platform and hence a time-varying fleet size; for example, \citet{Ozkan_Ward_2020} models a ride-hailing system as an open two-sided matching system. Our model is closed and is most appropriate over time horizons in which the fleet size is approximately constant; over longer horizons, the diffusion approximation can be applied in a piecewise fashion by updating the fleet size between intervals. Demand can also be non-stationary, particularly in the presence of dynamic pricing. Although this paper focuses on time-homogeneous primitives, such non-stationarity can be captured through time-varying service rates at the single-server stations, and we expect the convergence results to extend to such a time-inhomogeneous setting, although the mathematical details are beyond the scope of this paper. In particular, second-order adjustments to the effective demand rates would appear as drift terms in the limit and naturally give rise to drift-rate control problems; see, e.g., \citet{Alwan_et_al_2024}. Extending the convergence results to open networks with external arrivals and departures, while allowing jobs to recirculate through the network, together with infinite-server travel-time nodes, is left for future work.

\subsection{Literature Review}\label{Section: 1.1}

This paper contributes to the literature on heavy traffic diffusion approximations for queueing networks, with a particular focus on closed systems involving infinite-server stations. Classical work by \citet{Iglehart_1965}, \citet{Borovkov_1967}, and \citet{Whitt_Glynn_1991} establishes diffusion limits for open queueing networks with infinite-server queues. Other papers have studied the structure and properties of these limiting processes. For example, \citet{Glynn_1982} establishes necessary and sufficient conditions under which the heavy-traffic limit of the $GI/G/\infty$ queue is Markovian. Foundational work on heavy-traffic approximations for stochastic flow networks with bottlenecks includes \citet{Chen_Mandelbaum_1991_Fluid, Chen_Mandelbaum_1991_Diffusion}, which develop both fluid and diffusion limits. Various surveys, such as \citet{Glynn_1990} and \citet{Pang_Talreja_Whitt_2007} discuss a variety of techniques for deriving diffusion approximations in queueing systems; see the references therein for a broader overview.

In the context of closed queueing networks with infinite-server stations, several papers establish diffusion limits under heavy traffic conditions. Notably, \citet{Kogan_Lipster_Smorodinskii_1986}, \citet{Krichagina_1992}, and \citet{Kogan_Lipster_1993} analyze models with many single-server stations and a single infinite-server station. These papers have a single probability routing vector and use martingale methods to prove convergence. Our paper supplements this line of research by allowing an arbitrary number of infinite-server stations and a two-level probabilistic routing structure. (This enables more realistic modeling of origin-destination dynamics and heterogeneous travel times, which are important to service systems such as ride-hailing.) Our analysis also differs methodologically: we use a continuous mapping approach and establish existence, uniqueness, and continuity of a nonlinear regulator mapping---results that are not readily implied by existing work. 
To the best of our knowledge, this is the first paper to rigorously establish a diffusion limit for a closed queueing network with both multiple infinite-server stations and a two-level routing structure. 

We note that the network studied here is a closed queueing network with Markovian routing and service times (i.e., a Gordon--Newell network) with a fixed number of jobs $n$, and therefore admits a product-form stationary distribution. Using this stationary distribution, however, requires computing the normalizing constant, i.e., the partition function $G(n)$, which is a sum over all ways of allocating the $n$ jobs across stations; the number of such allocations grows combinatorially with the number of jobs and stations. This quickly becomes computationally difficult to evaluate in large networks. While this challenge has motivated work on asymptotic and algorithmic methods for approximating $G(n)$ in large closed networks, exact computation can still be burdensome when both the number of jobs and the number of stations are large; see, e.g., \citet{Birman_Kogan_1992} and \citet{Kogan_1992}. In contrast, the diffusion limit developed here provides a tractable process-level approximation for $\sqrt{n}$-scale fluctuations, capturing transient behavior and serving as a natural basis for diffusion-scale control and optimization in heavy traffic.

A closely related and important antecedent of our model is the one developed in \citet{Alwan_et_al_2024}, which considers a ride-hailing system modeled as a closed queueing network with multiple single-server regions and a single infinite-server travel node. The objective is to maximize system profit by making dynamic pricing and dispatch control decisions. That paper formulates an approximating Brownian control problem (BCP) using formal limiting arguments in the spirit of \citet{Harrison_1988, Harrison_2003}. Under a complete resource pooling assumption, the BCP is reduced to a one-dimensional equivalent workload formulation, which is then solved analytically. However, the assumption of a single infinite-server node effectively enforces homogeneous travel times and origin-independent destination routing. This modeling simplification was necessary to permit a one-dimensional state space collapse, but such a reduction is not possible when the model includes multiple infinite-server nodes.

In contrast, our model accommodates multiple infinite-server stations and a two-level routing structure that allows for heterogeneous travel times and captures general origin-destination routing distributions. (Notably, \citet{Braverman_et_al_2019} considers a related model with empty-car routing control but studies it under fluid scaling, which avoids the complexities associated with diffusion-scaled fluctuations.) Our setting leads to a multidimensional diffusion limit. 
By rigorously proving a weak convergence result in this setting, our work provides the theoretical foundation for studying more complex control policies and generalizes the modeling framework introduced in \citet{Alwan_et_al_2024}. 

As mentioned earlier, the presence of multiple infinite-server stations does not permit a one-dimensional state space collapse, implying that the limiting diffusion control problem remains high-dimensional. However, recent advancements in computing power have renewed interest in studying high-dimensional diffusion control problems using approximate dynamic programming and neural network-based methods. One expects the dynamic control problem associated with the system here to involve drift and singular control of (reflecting) Brownian motion in the heavy-traffic limit. Such problems have been successfully addressed in dimensions up to 100 or more in recent work by \citet{Ata_Harrison_Si_2024, Ata_Harrison_Si_2025}, and \citet{Ata_Xu_2025}. While those papers consider different applications, we expect their computational methods to be applicable to ride-hailing applications as well. Although the study of dynamic control of ride-hailing systems in high dimensions is left for future work, our paper lays the mathematical groundwork for future research in stochastic control of such systems, including those where dimensionality precludes analytical tractability.

Beyond ride-hailing, our results are also relevant to other applications modeled using queueing networks with infinite-server stations. For example, \citet{Ata_Tongarlak_Lee_Field_2024} studies nonprofit volunteer management using a queueing network with a single-server station and multiple infinite-server stations. They formulate an approximating Brownian control problem but lack rigorous justifications for their approximations. Our results help close this gap by providing provable limit theorems in a similar setting.

The rest of the paper is organized as follows. Section \ref{Section:2} introduces the model primitives and makes a sample-path construction of the queue-length processes describing the network. Section \ref{Section:3} articulates the heavy traffic assumption and the asymptotic regime, and states the main result of the paper (Theorem \ref{Theorem:Main_Result}). Section \ref{Section:4} develops the key tools needed to prove the main result. Section \ref{Section:5} is devoted to a proof of the main result, which involves proving weak convergence of the fluid-scaled processes. Relevant notation and technical preliminaries are given shortly. Proofs deferred from the main text are provided in Appendices \ref{appendix:C}--\ref{appendix:D}.\vspace{0.5em}

\textbf{Notation and Technical Preliminaries.} The set of positive integers is denoted by $\mathbb{N}:=\{1,2,3,\dots\}$, and we define $[k]\coloneqq \{1,2,\dots, k\}$ for $k\in\mathbb{N}$. For $a,b\in\mathbb{R}$, we let $a\vee b:= \max\{a,b\}$ and $a\wedge b := \min\{a,b\}$, and let $\lfloor a\rfloor$ denote the largest integer less than or equal to $a$. For $l\in [k]$, the $l$th unit basis vector in $\mathbb{R}^k$ is denoted by $e_l$, which has one in the $l$th component and zeros elsewhere. Moreover, for $l\in [k]$, the $l$th projection map $\pi_l:\mathbb{R}^k\rightarrow \mathbb{R}$ is given by $\pi_l(x) = x_l$, where $x_l$ is the $l$th component of $x\in \mathbb{R}^k$. For $k\in\mathbb{N}$, the positive orthant in $\mathbb{R}^k$ is denoted by $\mathbb{R}^k_+:=\{x\in\mathbb{R}_k:x_l\ge 0\text{ for all }l\in[k]\}$. For a function $f:X\rightarrow Y$ and a subset $S\subseteq X$, we denote by $f\vert_S$ the restriction of $f$ to $S$. The indicator function of a subset $S\subseteq X$ is denoted by $\mathds{1}_S$.

For $k\in\mathbb{N}$, we denote by $D^k\equiv D\left([0,\infty),\mathbb{R}^k\right)$ the set of all functions $x:[0,\infty)\rightarrow\mathbb{R}^k$ that are right continuous on $[0,\infty)$ and have left limits on $(0,\infty)$. We denote by $C^k\equiv C\left([0,\infty),\mathbb{R}^k\right)$ the set of all functions $x:[0,\infty)\rightarrow\mathbb{R}^k$ that are continuous. The identically zero function is denoted by $\mathbf{0}$. Similarly, for $k\in\mathbb{N}$ and $T>0$, we denote by $D_T^k\equiv D\left([0,T], \mathbb{R}^k\right)$ the set of all functions $x:[0,T]\rightarrow \mathbb{R}^k$ that are right continuous on $[0,T)$ and have left limits on $(0,T]$. When the space $D_T^k$ is endowed with the norm
\begin{align*}
\qquad\|x\|_{T,k}\coloneqq\,  \max_{l\in [k]}\sup_{t\in [0,T]}\vert x_l(t)\vert,
\end{align*}
it is a Banach space. 
When $k=1$, we write $D^1 = D$, $D_T^1 = D_T$, and $\|\cdot \|_{T,1} = \|\cdot \|_T$. The one-sided reflection map on $D$ is the pair of functions $(\psi, \phi):D\rightarrow D^2$ defined as follows:
\begin{alignat}{2}
\psi(x)(t)&\coloneqq \sup_{s\in [0,t]}[-x(s)]^+,&&\quad t\ge 0,\label{eq:D.2}\\
\phi(x)(t)&\coloneqq  x(t) \,+\, \psi(x)(t),&&\quad t\ge 0.\label{eq:D.3}
\end{alignat}
The space $D^k$ is a topological space when endowed with the Skorokhod $J_1$ topology. All random variables in this paper are defined on a common probability space $(\Omega, \mathcal{F}, P)$. We denote by $\mathcal{M}^k$ the Borel $\sigma$-algebra on $D^k$ induced by the Skorokhod $J_1$ topology. All stochastic processes in this paper are measurable functions from $(\Omega, \mathcal{F}, P)$ to $(D^k,\mathcal{M}^k)$ for some appropriate dimension $k$. For a sequence of stochastic processes $\{\xi^n\}_{n=1}^{\infty}$ and a stochastic process $\xi$, all with sample paths in $D^k$ almost surely, we write $\xi^n\Rightarrow \xi$ as $n\rightarrow\infty$ to mean that the sequence of probability measures on $(D^k,\mathcal{M}^k)$ induced by the processes $\xi^n$ converge weakly to the probability measure on $(D^k,\mathcal{M}^k)$ induced by the stochastic process $\xi$; see \citet{Billingsley_1999} and \citet{Whitt_2002} for further details.

\section{Queueing Network Model}\label{Section:2}
This section describes the closed queueing network model. In Section \ref{Section:2.1}, we introduce the system primitives. Then, in Section \ref{Section:2.2}, we describe the system performance measures and their dynamics.

\subsection{Model Primitives}\label{Section:2.1}
We consider a closed queueing network consisting of a fixed number of jobs circulating among $J$ single-server stations and $K$ infinite-server stations. Each single-server station $j\in [J]$ has a buffer where jobs of that class are stored and a single server with service rate of $\mu_j>0$. Each infinite-server station $k\in [K]$ has an unlimited number of parallel servers, each with a service rate of $\eta_k>0$. For future analysis, it is useful to define $\eta\coloneqq  \max_{k\in [K]}\eta_k$.

Jobs move through the system according to a two-level probabilistic routing structure: from single-server stations to infinite-server stations, and from infinite-server stations back to single-server stations. When a job completes service at single-server station $j\in [J]$, it is routed to infinite-server station $k\in [K]$ with probability $p_{jk}\in [0,1]$. Similarly, when a job completes service at infinite-server station $k\in [K]$, it is routed to single-server station $j\in [J]$ with probability $q_{kj}\in [0,1]$. To formalize this, we take as given stochastic matrices 
\begin{align}
\!\!\!\!P\,=\,(p_{jk})\,\in\, \mathbb{R}^{J\times K}\quad \text{and}\quad Q\,=\,(q_{kj})\,\in\,\mathbb{R}^{K\times J},\label{Equation:Stochastic_Matrix}
\end{align}
representing routing from single-server stations to infinite-server stations and from infinite-server stations to single-server stations, respectively.

The stochastic evolution of the system is driven by the following primitives. For $j\in [J]$, let $N_j = \{N_j(t):t\ge 0\}$ be a unit-rate Poisson process governing service completions at single-server station $j$. For $k\in [K]$, let $M_k =\{M_k(t) : t\ge 0\}$ be a unit-rate Poisson process governing service completions at infinite-server station $k$. To model the probabilistic routing of jobs in the network, let $\phi_j = \{\phi_j(l): l\ge 1\}$ for $j\in [J]$ and $\psi_k = \{\psi_k(l): l\ge 1\}$ for $k\in [K]$ denote sequences of i.i.d. random (routing) vectors. Their probability distributions are given by
\begin{align*}
\!\!\!\!P(\phi_j(1) \,=\, e_k) \,=\, p_{jk}\quad \text{and}\quad P(\psi_k(1) \,=\, e_j) \,=\, q_{kj},
\end{align*}
where $e_k$ and $e_j$ are the $k$th and $j$th standard unit basis vector in $\mathbb{R}^K$ and $\mathbb{R}^J$, respectively. We assume that all stochastic primitives are mutually independent. 

To track the cumulative routing of jobs, for $j\in [J]$ and $k\in [K]$, we define the cumulative routing processes $\Phi_{jk} = \{\Phi_{jk}(m):m\ge 1\}$ and $\Psi_{kj} = \{\Psi_{kj}(m): m\ge 1\}$ as follows:
\begin{align}
\Phi_{jk}(m) \coloneqq  \sum_{l=1}^{m}\phi_{jk}(l),\quad \Psi_{kj}(m) \coloneqq  \sum_{l=1}^{m}\psi_{kj}(l),\label{eq:2.6}
\end{align}
where $\phi_{jk}(l)$ and $\psi_{kj}(l)$ are the $k$th and $j$th components of $\phi_{j}(l)$ and $\psi_k(l)$, respectively. That is, $\Phi_{jk}(m)$ represents the total number of jobs that are routed from single-server station $j$ to infinite-server station $k$ among the first $m$ jobs served by server $j$. The interpretation of $\Psi_{kj}(m)$ is similar.

\subsection{State Dynamics}\label{Section:2.2}

For $j\in [J]$ we denote by $Q_j(t)$ the number of jobs in the buffer of the $j$th single-server station at time $t$. Similarly, for $k\in [K]$ we denote by $V_k(t)$ the number of jobs in the $k$th infinite-server station at time $t$. To describe the state dynamics of $Q_j$ and $V_k$, let $Q_j(0)$ for $j\in [J]$ and $V_k(0)$ for $k\in [K]$ be almost surely nonnegative random variables representing the initial distribution of the number of jobs in the network. We assume that these random variables are independent of all other stochastic primitives. Then, for $j\in [J]$ and $k\in [K]$, we have that
\begin{alignat}{3}
Q_j(t) &\,=\, Q_j(0) \,+\, A_j(t) \,-\, D_j(t),&&\,\,\,\,\, t\ge 0,\label{eq:2.8}\\
V_k(t) &\,=\, V_k(0) \,+\, E_k(t) \,-\, F_k(t),&&\,\,\,\,\, t\ge 0,\label{eq:2.9}
\end{alignat}
where $A_j =\{A_j(t): t\ge 0\}$ and $D_j = \{D_j(t): t\ge 0\}$ denote the arrival and departure processes for single-server station $j$. Similarly, $E_k = \{E_k(t): t\ge 0\}$ and $F_k = \{F_k(t):t\ge 0\}$ denote the arrival and departure processes for infinite-server station $k$. 
To be more specific, the arrival and departure processes are defined as follows:
\begin{alignat}{2}
A_j(t) &\coloneqq  \sum_{k=1}^{K}\Psi_{kj}\big(F_k(t)\big),&&\quad t\ge 0,\label{eq:2.10}\\[-1mm]
E_k(t) &\coloneqq  \sum_{j=1}^J \Phi_{jk}\big(D_j(t)\big),&&\quad t\ge 0,\label{eq:2.11}\\[0.5mm]
D_j(t) &\coloneqq  N_j\big(\mu_j T_j(t)\big),&&\quad t\ge 0,\label{eq:2.12}\\[1mm]
F_k(t) &\coloneqq  M_k \Big(\eta_k \!\int_{0}^{t}\!V_k(s)\,ds\Big),&&\quad t\ge 0,\label{eq:2.13}
\end{alignat}
where $\Phi_{jk}$ and $\Psi_{kj}$ are given by (\ref{eq:2.6}), and $T_j = \{T_j(t):t\ge 0\}$
is a process that represents the scheduling policy for the server at the $j$th single-server station. To be more specific, $T_j(t)$ is the cumulative amount of time the server is working up to time $t$ at single-server station $j$. The corresponding idleness process $I_j = \{I_j(t):t\ge 0\}$ for $j\in [J]$ of the server at single-server station $j$ is defined as follows:
\begin{align}
\,\,\,\,\,\,\,\,\,\,\,\,\,\,\,I_j(t)\coloneqq\,  t \,-\, T_j(t),\quad t\ge 0.\label{eq:2.15}
\end{align}
By (\ref{eq:2.8})--(\ref{eq:2.13}), it is straightforward to verify that
\begin{align}
\!\!\!\sum_{j=1}^{J}Q_j(t) \,+ \sum_{k=1}^{K}V_k(t)\,=\sum_{j=1}^{J}Q_j(0) \,+ \sum_{k=1}^{K}V_k(0),\label{eq:2.16}
\end{align}
almost surely for all $t\ge 0$. Throughout our analysis, we restrict to scheduling policies $T$ that satisfy the following conditions almost surely for all $j\in [J]$:
\begin{align}
&\!\!\!\!\!\!\!\!\!\!Q_j(t)\in [0,\infty)\,\text{ for all }\,t\ge 0,\label{eq:2.18}\\[1pt]
&\!\!\!\!\!\!\!\!\!\!\int_{0}^{\infty}\!\mathds{1}_{\{Q_j(t) \,>\,0\}}\,dI_j(t) \,=\, 0,\label{eq:2.17}\\
&\!\!\!\!\!\!\!\!\!\!I_j \text{ is continuous and nondecreasing with } I_j(0)=0,\label{eq:2.19}\\[5pt]
&\!\!\!\!\!\!\!\!\!\!I_j(t) - I_j(s) \,\le\, t-s\,\text{ for all }\,0\le s\le t<\infty.\label{eq:2.20}
\end{align}
Equation (\ref{eq:2.18}) reflects the obvious physical restriction that the servers can only work on jobs when the queues are not empty. Equation (\ref{eq:2.17}) implies that the server idleness at each single-server station does not increase as long as the queue is not empty. In other words, we restrict attention to (otherwise arbitrary) work-conserving scheduling policies. Finally, (\ref{eq:2.19})--(\ref{eq:2.20}) are natural consequences of the interpretation of $I_j$ as server idleness and are standard in the heavy traffic literature. 

To lighten the notation, we henceforth work pathwise on a full-measure set on which the initial conditions, the scheduling-policy constraints, and the primitives' usual path properties hold for all $t\ge 0$; we therefore omit the ``almost surely'' terminology throughout. 

\section{Heavy Traffic Assumption and Main Result}\label{Section:3}

As is standard in heavy traffic asymptotic analysis, we consider a sequence of queueing networks---as described in Section \ref{Section:2}---indexed by the system parameter $n$. We study this sequence of systems in the heavy traffic asymptotic regime as $n\rightarrow\infty$. Throughout, we attach a superscript of $n$ to the various quantities of interest to indicate that they correspond to the $n$th system. 

We consider a regime in which both the number of jobs and the service rates at the single-server stations grow large with the system parameter $n$, and where the system satisfies a critical loading condition. However, we assume that the service rates at the infinite-server stations do not vary with the system parameter $n$. This regime is formalized by the following heavy traffic assumption. To state it, define
\begin{align}
    m_k\coloneqq\,  \eta_k^{-1}\sum_{j=1}^{J}\mu_j p_{jk},\quad k\in [K].
\end{align}
\begin{assumption}[Heavy Traffic Assumption]\label{Assumption:Heavy_Traffic}
	The service rates at the single-server stations and infinite-server stations, respectively, vary with $n$ as follows: $\mu_j^n =  n\mu_j$ for $j\in [J]$ and $\eta^n_k=  \eta_k$ for $k\in [K]$. Moreover, the following critical loading conditions hold:
	\begin{align}
	&\sum_{k=1}^{K}\sum_{i=1}^{J}\mu_ip_{ik}q_{kj} \,=\, \mu_j\quad \text{for all}\quad j\in [J],\label{HT1}\\
	&\sum_{k=1}^{K}m_k\,=\,1.\label{HT2}
	\end{align}
\end{assumption}

Roughly speaking, condition (\ref{HT1}) assumes that every single-server is fully utilized, whereas (\ref{HT2}) roughly says that almost all jobs are in the infinite-server stations; see \citet{Alwan_et_al_2024} and \citet{Ata_Tongarlak_Lee_Field_2024} for similar assumptions in ride-hailing and volunteer engagement applications, respectively. For simplicity, Assumption \ref{Assumption:Heavy_Traffic} focuses on the zero-drift case. However, this assumption can be refined to give a nonzero drift term in the heavy-traffic limit. 
In particular, suppose that the service rates at the single-server stations in the sequence satisfy
$$\sqrt{n}\,\Big[\sum_{k=1}^{K}\sum_{i=1}^{J}\mu_ip_{ik}q_{kj}\,-\, n^{-1}\mu_j^n\Big]\rightarrow c_j\in\mathbb{R} \quad\text{as}\quad n\rightarrow\infty,$$ for all $j\in [J]$. Under this condition, our main result, i.e., Theorem \ref{Theorem:Main_Result}, remains valid, except that the limiting diffusion for the single-server queue-length process acquires an additional drift term $c_jt$.

To shed light on (\ref{HT1}), note that $\sum_{i=1}^{J}n\mu_i p_{ik}$ represents the total rate of jobs leaving the single-server stations to infinite-server station $k$. Thus, $\sum_{k=1}^{K}q_{kj}\sum_{i=1}^{J}n\mu_i p_{ik}$ represents the total rate of jobs entering single-server station $j$ from the infinite-server stations. Because $n\mu_j$ is the service rate of the server in single-server station $j$, (\ref{HT1}) states that the rate of jobs entering the single-server stations are balanced out by the service rates at the stations.

To shed light on (\ref{HT2}), note that $n\sum_{j=1}^{J}\mu_j p_{jk}$ represents the arrival rate to the $k$th infinite-server station, whereas its service rate is $\eta_k$. Based on intuition from the classical $M/M/\infty$ queue, we expect the steady-state average queue length at the $k$th infinite-server to be $n\sum_{j=1}^{J}\mu_jp_{jk}/\eta_k$. It follows that the expected fraction of jobs at the $k$th infinite-server is $m_k$. Therefore, (\ref{HT2}) states that almost all jobs are at the infinite-server stations. This is consistent with the first condition: It is well known that in large balanced-flow systems in heavy traffic, the queue-length processes are of second-order relative to the system size. Thus, the total number of jobs in the single-server stations are of order $\sqrt{n}$, which implies that the total number of jobs at the infinite-server stations are of order $n$, since the network is closed.

To facilitate the analysis to follow, we next define the following diffusion- and fluid-scaled processes:\vspace{1em}

\noindent\textbf{Diffusion-Scaled Processes:} For $j\in [J]$ and $k\in [K]$, we define the following diffusion-scaled processes:
\begin{alignat}{2}
\!\!\!\hat{Q}_j^n(t) &\coloneqq\,  n^{-1/2}Q_j^n(t),&&\quad t\ge 0,\label{eq:3.1}\!\\
\!\!\!\hat{V}_k^n(t) &\coloneqq\,  n^{-1/2}\big(V_k^n(t) \,-\, nm_k\big),&&\quad t\ge 0,\label{eq:3.2}\!\\
\!\!\!\hat{I}^n_j(t) &\coloneqq\,  \sqrt{n} I^n_j(t),&&\quad t\ge 0,\label{eq:3.3}\!\\
\!\!\!\hat{T}^n_j(t)&\coloneqq\, \sqrt{n} T^n_j(t),&&\quad t\ge 0,\label{eq:3.4}\!\\
\!\!\!\hat{\Phi}^n_{jk}(t) &\coloneqq\, n^{-1/2}\big(\Phi_{jk}(\lfloor nt\rfloor)\,-\,p_{jk}nt\big),&&\quad t\ge 0,\label{eq:3.5}\!\\
\!\!\!\hat{\Psi}_{kj}^n(t) &\coloneqq\,  n^{-1/2}\big(\Psi_{kj}(\lfloor nt\rfloor) \,-\, q_{kj}nt\big),&&\quad t\ge 0,\label{eq:3.6}\!\\
\!\!\!\hat{N}^n_j(t)&\coloneqq\,  n^{-1/2}\big(N_j(nt)\,-\,nt\big),&&\quad t\ge 0,\label{eq:3.7}\!\\
\!\!\!\hat{M}^n_k(t)&\coloneqq\,  n^{-1/2}\big(M_k(nt)\,-\, nt\big),&&\quad t\ge 0.\label{eq:3.8}\!
\end{alignat}
\noindent\textbf{Fluid-Scaled Processes:} For $j\in [J]$ and $k\in [K]$, we define the following fluid-scaled processes:
\begin{alignat}{2}
\bar{Q}_j^n(t) &\coloneqq\,  n^{-1}Q_j^n(t),&&\quad t\ge 0,\label{eq:3.9}\\
\bar{V}_k^n(t) &\coloneqq\,  n^{-1/2}\hat{V}_k^n(t)&&\quad t\ge 0,\label{eq:3.10}\\
\rttensortwo{V}_k^n(t)&\coloneqq\,  n^{-1}V_k^n(t),&&\quad t\ge 0,\label{eq:3.11}\\
\bar{N}^n_j(t)&\coloneqq\,  n^{-1}N_j(nt),&&\quad t\ge 0\label{eq:3.14}\\
\bar{M}^n_k(t)&\coloneqq\,  n^{-1}M_k(nt),&&\quad t\ge 0.\label{eq:3.15}
\end{alignat}
By (\ref{eq:2.8})--(\ref{eq:2.9}), (\ref{eq:3.1})--(\ref{eq:3.15}), and Assumption \ref{Assumption:Heavy_Traffic}, it is straightforward verify that for all $j\in [J]$, $k\in [K]$, and $t\ge 0$ the following equalities hold:
\begin{align}
\!\!\!\!\!\!\!\!\hat{Q}_j^n(t) \,&=\, \hat{\xi}_j^n(t)\,+ \sum_{k=1}^{K}q_{kj}\eta_k\!\int_{0}^{t}\!\hat{V}_k^n(s)\,ds \,+\, \mu_j\hat{I}_j^n(t),\label{eq:3.20}\!\!\\
\!\!\!\!\!\!\!\!\hat{V}_k^n(t) \,&=\, \hat{\zeta}_k^n(t)\,-\,\eta_k \!\int_{0}^{t}\!\hat{V}_k^n(s)\,ds \,- \sum_{j=1}^{J}p_{jk}\mu_j\hat{I}_j^n(t),\label{eq:3.21}\!\!
\end{align}
where
\begin{align}
\!\!\!\!\!\!\!\!\!\!\!\!\!\!\!\hat{\xi}_j^n(t)&\coloneqq\, \hat{Q}_j^n(0)  \,+ \sum_{k=1}^{K}\hat{\Psi}^n_{kj}\Big(\bar{M}^n_k\big(\eta_k\!\int_{0}^{t}\!\rttensortwo{V}_k^n(s)\,ds\big)\Big) \notag\\&\qquad\!-\, \hat{N}^n_j\big(\mu_j T_j^n(t)\big)\,+ \sum_{k=1}^{K}q_{kj}\hat{M}^n_k\Big(\eta_k\!\int_{0}^{t}\!\rttensortwo{V}_k^n(s)\,ds\Big),\label{eq:3.18}\\
\!\!\!\!\!\!\!\!\!\!\!\!\!\!\!\hat{\zeta}_k^n(t)&\coloneqq\, \hat{V}_k^n(0) \,+ \sum_{j=1}^{J}\hat{\Phi}^n_{jk}\big(\bar{N}^n_j(\mu_j T_j^n(t))\big) \notag\\&\qquad\!-\, \hat{M}^n_k \Big(\eta_k\!\int_{0}^{t}\!\rttensortwo{V}_k^n(s)\,ds\Big) \,+ \sum_{j=1}^{J}p_{jk}\hat{N}^n_j(\mu_j T_j^n(t)).\label{eq:3.19}
\end{align}

Moreover, by (\ref{eq:3.9})--(\ref{eq:3.10}) and (\ref{eq:3.20})--(\ref{eq:3.21}), it is straightforward to verify that for all $j\in [J]$, $k\in [K]$, and $t\ge 0$ the following equalities hold:
\begin{align}
\!\!\!\!\!\!\!\!\!\bar{Q}_j^n(t) \,&=\, \bar{\xi}_j^n(t)\,+ \sum_{k=1}^{K}q_{kj}\eta_k\!\int_{0}^{t}\!\bar{V}_k^n(s)\,ds \,+\, \mu_jI_j^n(t),\label{eq:3.22}\!\!\\
\!\!\!\!\!\!\!\!\!\bar{V}_k^n(t) \,&=\, \bar{\zeta}_k^n(t)\,-\,\eta_k\! \int_{0}^{t}\!\bar{V}_k^n(s)\,ds \,- \sum_{j=1}^{J}p_{jk}\mu_jI_j^n(t),\label{eq:3.23}\!\!
\end{align}
where 
\begin{align}
\!\!\!\!\!\bar{\xi}_j^n(t)\coloneqq\,  n^{-1/2}\hat{\xi}_j^n(t)\quad\text{and}\quad
\bar{\zeta}_k^n(t)\coloneqq\,  n^{-1/2}\hat{\zeta}_k^n(t).\label{eq:Fluid_Free_Processes}
\end{align}

We make the following regularity assumption on the initial conditions:

\begin{assumption}[Joint Convergence of the Initial Conditions]\label{Assumption_Initial_Conditions}
	As $n\rightarrow\infty$, $(\hat{Q}^n(0),\hat{V}^n(0))\Rightarrow (Q(0),V(0))$.
\end{assumption}

To facilitate the statement of our main result, let $(\xi^*,\zeta^*)$ be a $(J+K)$-dimensional Brownian motion with initial state 
$(Q(0), V(0))$ and covariance matrix $\Sigma\in\mathbb{R}^{(J+K)\times (J+K)}$, where $\Sigma$ is given as follows: For $i,j\in [J]$ and $k,l\in [K]$ such that $i\ne j$ and $k\ne l$, we have
\begin{alignat}{2}
&\Sigma_{j,j}&&\,=\,\,\,\,\, \sum_{k=1}^{K}q_{kj}(1-q_{kj})\eta_km_k\notag\\[-0.75em]&&&\qquad\qquad+\,\mu_j \,+\sum_{k=1}^{K}q_{kj}^2\eta_km_k,\label{eq:3.29*} \\
&\Sigma_{J+k,J+k}&&\,=\,\,\,\,\, \sum_{j=1}^{J}p_{jk}(1-p_{jk})\mu_j\notag\\[-0.75em]&&&\qquad\qquad+\,\eta_k m_k\,+\sum_{j=1}^{J}p_{jk}^2\mu_j,\\
&\Sigma_{i,j} &&\,=\,\,\,\,\, \sum_{k=1}^{K}q_{ki}q_{kj}\eta_km_k, \\
&\Sigma_{j,J+k}&&\,=\,\,\,\,\,-p_{jk}\mu_j\,-\,q_{kj}\eta_km_k,\\
&\Sigma_{J+l,J+k}&&\,=\,\,\,\,\,\sum_{j=1}^{J}p_{jl}p_{jk}\mu_j.\label{eq:3.33*}
\end{alignat}

\begin{theorem}
\label{Theorem:Main_Result}
	As $n\rightarrow\infty$, $(\hat{Q}^n,\hat{I}^n,\hat{V}^n)\Rightarrow (Q^*, I^*, V^*)$,
	where $(Q^*,I^*, V^*)$ is a $(2J+K)$-dimensional process with continuous sample paths in $\mathbb{R}^{2J}_+\times \mathbb{R}^K$ that satisfies the following equalities for all $j\in [J]$, $k\in [K]$, and $t\ge 0$:
	\begin{align}
	&\!\!\!\!\!\!\!\!\!\!Q_j^*(t) \,=\, \xi^*_j(t) \,+ \sum_{k=1}^{K}q_{kj}\eta_k\!\int_{0}^{t}\!V_k^*(s)\,ds \,+\, \mu_j I_j^*(t),\label{eq:3.27}\\
	&\!\!\!\!\!\!\!\!\!\!V_k^*(t)\,=\, \zeta^*_k(t)\,-\,\eta_k\!\int_{0}^{t}V_k^*(s)\,ds\,-\sum_{j=1}^{J}p_{jk}\mu_j I_j^*(t),\label{eq:3.28}\\
    &\!\!\!\!\!\!\!\!\!\!I_j^*(t)\,=\,\mu_j^{-1}\psi\Big(\xi^*_j\,+\sum_{k=1}^{K}q_{kj}\eta_k\!\int_{0}^{\cdot}\!V_k^*\,ds \Big)(t),\\
	&\!\!\!\!\!\!\!\!\!\!\int_{0}^{\infty}\!\mathds{1}_{\{Q_j^*(t)\,>\,0\}}\,dI_j^*(t)\,=\,0.\label{eq:3.29}
	\end{align}
\end{theorem}

\section{Auxiliary Results}\label{Section:4}
This section establishes the existence of (suitably defined) continuous functions that will aid in the proof of Theorem~\ref{Theorem:Main_Result} via a continuous mapping argument. To that end, let $\xi\in D^J$ and $\zeta\in D^K$ be functions such that
\begin{alignat}{2}
&\sum_{j=1}^{J}\xi_j(t) \,+ \sum_{k=1}^{K}\zeta_k(t)\,=\,0,&&\quad \text{for all } t\ge 0,\label{eq:4.1}\\
&\xi_j(0)\,\ge\, 0,&&\quad\text{for all } j\in [J].\label{eq:4.2}
\end{alignat}
Given such functions $\xi$ and $\zeta$, 
consider the following system of equations 
for $j\in [J]$, $k\in [K]$, and $t\ge 0$:
\begin{align}
&\!\!\!\!\!\!\!\!\!\!\!\!\!\!x_j(t) \,=\, \xi_j(t) \,+ \sum_{k=1}^{K}q_{kj}\eta_k\!\int_{0}^{t}\!y_k(s)\,ds \,+\, \mu_j u_j(t)\,\ge\, 0,\label{eq:4.3}\\[-5pt]
&\!\!\!\!\!\!\!\!\!\!\!\!\!\!y_k(t) \,=\, \zeta_k (t) \,-\, \eta_k\! \int_{0}^{t}\!y_k(s)\,ds \,- \sum_{j=1}^{J}p_{jk}\mu_ju_{j}(t),\label{eq:4.4}\\[-7pt]
&\!\!\!\!\!\!\!\!\!\!\!\!\!\!\sum_{j=1}^{J}x_j(t) \,+ \sum_{k=1}^{K}y_k(t) \,=\, 0,\label{eq:4.5}\\
&\!\!\!\!\!\!\!\!\!\!\!\!\!\!u_j \text{ is nondecreasing with } u_j(0) = 0,\label{eq:4.6}\\[3pt]
&\!\!\!\!\!\!\!\!\!\!\!\!\!\!\int_{0}^{\infty} \!\mathds{1}_{\{x_j(t)\,>\,0\}}\,du_j(t)\,=\,0.\label{eq:4.7}
\end{align}
The following result establishes the existence and uniqueness of a triple $(x,u,y)$ satisfying the above equations.
\begin{proposition}\label{Prop:4.1}
	For every $(\xi, \zeta) \in D^{J+K}$ satisfying (\ref{eq:4.1})--(\ref{eq:4.2}), 
	there exists a unique $(x, u, y)\in D^{2J+K}$ satisfying (\ref{eq:4.3})--(\ref{eq:4.7}).
\end{proposition}
\begin{proof}
    See Appendix \ref{appendix:A}.
\end{proof}

The next result is immediate from Proposition \ref{Prop:4.1}.
\begin{corollary}\label{Cor:4.2}
	There exists a function $f:D^{J+K}\rightarrow D^{2J+K}$ such that whenever $(\xi,\zeta)\in D^{J+K}$ satisfies (\ref{eq:4.1})--(\ref{eq:4.2}), $f(\xi, \zeta)\in D^{2J+K}$ satisfies (\ref{eq:4.3})--(\ref{eq:4.7}).
\end{corollary}

The following result is useful in the proof of Proposition~\ref{Prop:4.1}. To state it, given functions $\xi\in D^J$ and $\zeta\in D^{K}$, consider the following equation for $k\in [K]$ and $t\ge 0$: 
\begin{align}
\!\!\!\!\!\!\!\!\!\!\!\!y_k(t) \,&=\, \zeta_k(t) \,-\,\eta_k\!\int_{0}^{t}\!\!y_k(s)\,ds \nonumber\\&\qquad- \sum_{j=1}^{J}p_{jk}\psi\Big(\xi_j \,+ \sum_{l=1}^{K}q_{lj}\eta_l\!\int_{0}^{\cdot}\!y_l(s)\,ds\Big)(t).\label{eq:A.1}
\end{align}
\begin{lemma}\label{Lemma:A.1}
	For each $(\xi, \zeta)\in D^{J+K}$, there exists a unique $y\in D^{K}$ satisfying (\ref{eq:A.1}).
\end{lemma}
\begin{proof}
    See Appendix \ref{appendix:C}.
\end{proof}

Below we provide a description of the function $f$ from Corollary \ref{Cor:4.2}. Let $f_3: D^{J+K}\rightarrow D^K$ be the mapping that sends $(\xi, \zeta)\in D^{J+K}$ to the unique $y\in D^K$ satisfying (\ref{eq:A.1}); see Lemma \ref{Lemma:A.1}. Then, following the proof of Proposition \ref{Prop:4.1}, let $f_1:D^{J+K}\rightarrow D^J$ and $f_2:D^{J+K}\rightarrow D^J$ be the mappings defined as follows:
\begin{align}
&\!\!\!\!\!\!\!\!\!\!\!\!\!\!\!f_1(\xi, \zeta)\notag\\&\!\!\!\!\!\!\!\!\!\!\!\!\!\!\coloneqq  \Big(\phi\big(\pi_j\circ\xi + \sum_{l=1}^{K}q_{lj}\eta_l\!\int_{0}^{\cdot}\!\!\left(\pi_l\circ f_3(\xi, \zeta)\right)\!(s)\,ds\big)\Big)_{j\in [J]},\label{eq:A.11}\\
&\!\!\!\!\!\!\!\!\!\!\!\!\!\!\!f_2(\xi, \zeta)\notag\\&\!\!\!\!\!\!\!\!\!\!\!\!\!\!\coloneqq  \Big(\mu_j^{-1}\psi\big(\pi_j\circ\xi + \sum_{l=1}^{K}q_{lj}\eta_l\!\int_{0}^{\cdot}\!\!\left(\pi_l\circ f_3(\xi, \zeta)\right)\!(s)\,ds\big)\Big)_{j\in [J]}.\label{eq:A.12}
\end{align}
Let $f:=(f_1,f_2,f_3)$. Then, whenever $(\xi,\zeta)\in D^{J+K}$ satisfies (\ref{eq:4.1})--(\ref{eq:4.2}), $f(\xi, \zeta)\in D^{2J+K}$ satisfies (\ref{eq:4.3})--(\ref{eq:4.7}). The next result establishes continuity of $f$.
\begin{proposition}\label{Prop:4.3}
	The function $f:D^{J+K}\rightarrow D^{2J+K}$ is continuous when both the domain and range are endowed with the Skorokhod $J_1$ topology.
\end{proposition}
\begin{proof}
    See Appendix \ref{appendix:B}.
\end{proof}

\begin{proposition}\label{Prop:Preserve_Continuity}
    The function $f$ maps $C^{J+K}$ into $C^{2J+K}$.
\end{proposition}
\begin{proof}
    Following the same argument as in the proof of Lemma~\ref{Lemma:A.1}, but now with $(\xi,\zeta)\in C^{J+K}$, we can construct a sequence $\{y^n:n=0,1,\dots\}$ in $C^K$ (via the method of successive approximations) whose limit $y\in C^K$ is the unique solution to (\ref{eq:A.1}). It follows that $f_3(C^{J+K})\subseteq C^{K}$. It then follows from (\ref{eq:A.11}) and (\ref{eq:A.12}), together with the fact that the operations involved preserve continuity, that $f_1(C^{J+K})\subseteq C^{J}$ and $f_2(C^{J+K})\subseteq C^{J}$, respectively. Therefore, $f(C^{J+K})\subseteq C^{2J+K}$, as desired.
\end{proof}

\section{Main Convergence Results}\label{Section:5}
This section contains the main convergence results of this paper, culminating with a proof of Theorem~\ref{Theorem:Main_Result}. In Section~\ref{Section:5.1}, we prove convergence of the fluid scaled processes. (These results are necessary because several of the fluid scaled processes serve as random time changes in the diffusion-scaled equations.) In Section~\ref{Section:5.2}, we prove convergence of the process $(\hat{\xi}^n, \hat{\zeta}^n)$. This, combined with a continuous mapping argument, allows us to complete the proof of Theorem~\ref{Theorem:Main_Result}.

\subsection{Convergence of Fluid Scaled Processes}\label{Section:5.1}

We begin by establishing weak convergence of the fluid-scaled processes.

\begin{lemma}\label{Lemma:5.1}
	As $n\rightarrow\infty$, $(\bar{\xi}^n, \bar{\zeta}^n)\Rightarrow\mathbf{0}\in D^{J+K}$.
\end{lemma}
\begin{proof}
To prove that $(\bar{\xi}^n, \bar{\zeta}^n)\Rightarrow\mathbf{0}$ as $n\rightarrow\infty$, it suffices to show that $\bar{\xi}_j^n\Rightarrow 0$ as $n\rightarrow\infty$ for all $j\in [J]$ and $\bar{\xi}_k^n\Rightarrow 0$ as $n\rightarrow\infty$ for all $k\in[K]$; see, e.g., \citet[Theorem 11.4.5]{Whitt_2002}. In turn, it suffices to show that for all $T>0$, 
	\begin{align}
	\|\bar{\xi}_j^n\|_T\Rightarrow 0\quad\text{and}\quad \|\bar{\zeta}_k^n\|_T\Rightarrow 0\quad\text{as}\quad n\rightarrow\infty\label{eq:5.1}
	\end{align}
	for all $j\in [J]$ and $k\in [K]$; see Lemma \ref{Lemma:C.1} in Appendix \ref{appendix:D} for a proof of this claim. By (\ref{eq:3.18})--(\ref{eq:3.19}), the triangle inequality, and the fact that $\int_{0}^{t}\rttensortwo{V}_k^n(s)\,ds\le t$ and $T_j^n(t)\le t$ for all $t\ge 0$, it follows that for all $T>0$,
	\begin{align}
	\!\!\!\!\!\|\hat{\xi}_j^n\|_T&\,\le\, \|\hat{Q}_j^n(0) \|_T \,+ \sum_{k=1}^{K}\|\hat{\Psi}^n_{kj}\big(\bar{M}^n_k(\eta_k \halfspace\cdot )\big)\|_T \notag\\[-2pt]&\qquad+\, \|\hat{N}_j^n(\mu_j \halfspace\cdot)\|_T \,+ \sum_{k=1}^{K}\|\hat{M}_k^n(\eta_k\halfspace\cdot )\|_T,\label{eq:5.2}\\
	\!\!\!\!\!\|\hat{\zeta}_k^n\|_T&\,\le\, \|\hat{V}_k^n(0)\|_T \,+ \sum_{j=1}^{J}\|\hat{\Phi}_{jk}^n\big(\bar{N}_j^n(\mu_j\halfspace\cdot )\big)\|_T \notag\\[-2pt]&\qquad +\, \|\hat{M}_k^n(\eta_k\halfspace\cdot)\|_T \,+ \sum_{j=1}^{J}\|\hat{N}_j^n(\mu_j\halfspace\cdot)\|_T.\label{eq:5.3}
	\end{align}
	By Donsker's theorem, the functional central limit theorem for renewal processes, the random time change theorem, and the continuous mapping theorem, it is straightforward to show that the right-hand sides of (\ref{eq:5.2}) and (\ref{eq:5.3}) converge weakly to nondegenerate limits; see, e.g., \citet{Billingsley_1999} and \citet{Glynn_1990}. By this and the fact that $\bar{\xi}_j^n=n^{-1/2}\hat{\xi}_j^n$ and $\bar{\zeta}_k^n = n^{-1/2}\hat{\zeta}_k^n$, we obtain (\ref{eq:5.1}). This is a standard argument, so the detailed proof is omitted.\end{proof}

\begin{lemma}\label{Lemma:5.2}
	As $n\rightarrow\infty$, $(\bar{Q}^n, I^n, \bar{V}^n)\Rightarrow\mathbf{0}\in D^{J+K}$.
\end{lemma}
\begin{proof} 
	It is straightforward to show that the process $(\bar{\xi}^n, \bar{\zeta}^n)$ defined by (\ref{eq:Fluid_Free_Processes}) satisfies 
    \begin{alignat}{3}
	&\!\!\!\!\!\!\!\bar{\xi}^n_j(0) \,=\, \bar{Q}^n_j(0)\,\ge\, 0,&&\quad\text{for all } j\in [J],\label{Equation:Fluid_Xi_Nonnegative}\\
	&\!\!\!\!\!\!\!\sum_{j=1}^{J}\bar{\xi}_j^n(t) \,+ \sum_{k=1}^{K}\bar{\zeta}_k^n(t)\,=\,0,&&\quad\text{for all } t\ge 0.\label{Equation:Fluid_Sum_Zero}
	\end{alignat}
	Furthermore, by (\ref{eq:2.17})--(\ref{eq:2.19}) and (\ref{eq:3.9}), $I^n$ is nondecreasing componentwise with $I^n(0)=0$ and satisfies
	\begin{align}
	\!\!\!\!\!\!\!\!\!\!\!\!\!\int_{0}^{\infty}\!\!\mathds{1}_{\{\bar{Q}^n_j(t)\,>\,0\}}\,dI^n_j(t) &\,= \!\int_{0}^{\infty}\!\!\mathds{1}_{\{{Q}^n_j(t)\,>\,0\}}\,d{I}^n_j(t)\,=\,0.\label{eq:5.7}
	\end{align}
	Since $(\bar{\xi}^n, \bar{\zeta}^n)\in D^{J+K}$ pathwise, it follows from (\ref{eq:3.22})--(\ref{eq:3.23}), (\ref{Equation:Fluid_Xi_Nonnegative})--(\ref{eq:5.7}), and Proposition~\ref{Prop:4.1} that $(\bar{Q}^n, I^n,\bar{V}^n)=(f_1(\bar{\xi}^n, \bar{\zeta}^n), f_2(\bar{\xi}^n, \bar{\zeta}^n), f_3(\bar{\xi}^n, \bar{\zeta}^n))$. Then, by Proposition \ref{Prop:4.3}, Lemma \ref{Lemma:5.1}, and the continuous mapping theorem, it follows that as $n\rightarrow\infty$, 
	\begin{align*}
	(\bar{Q}^n, I^n,\bar{V}^n) &\,=\, \big(f_1(\bar{\xi}^n, \bar{\zeta}^n), f_2(\bar{\xi}^n, \bar{\zeta}^n), f_3(\bar{\xi}^n, \bar{\zeta}^n)\big)\\&\,\Rightarrow\, \big(f_1(\mathbf{0}), f_2(\mathbf{0}), f_3(\mathbf{0})\big)
	\end{align*}
	It now suffices to show that $\bar{V}\coloneqq f_3(\mathbf{0})=\mathbf{0}$, for then it follows from (\ref{eq:A.11})--(\ref{eq:A.12}) that $f_1(\mathbf{0})=\mathbf{0}$ and $f_2(\mathbf{0})=\mathbf{0}$. To that end, it follows from (\ref{eq:A.1}), the definition of $\bar{V}$, and the triangle inequality that, for any fixed $T>0$, all $t\in [0,T]$, and all $k\in [K]$, we have
	\begin{align*}
	\!\!\!\!\!\!\!\!\!\|\bar{V}_k\|_t
	&\,\le\, \eta\!\int_{0}^{t}\!\|\bar{V}_k\|_s\,ds \,+ \sum_{j=1}^{J}\sum_{l=1}^{K}\eta_l\Big\|\int_{0}^{\cdot}\!\bar{V}_l(s)\,ds\Big\|_t\notag\\
	&\,\le\, 2\eta JK\!\!\int_{0}^{t}\!\!\max_{k\in [K]}\|\bar{V}_k\|_s\,ds.
	\end{align*}
    Therefore, it follows that
	\begin{align}
	\max_{k\in [K]}\|\bar{V}_k\|_t\,\le\, 2\eta JK\!\!\int_{0}^{t}\!\!\max_{k\in [K]}\|\bar{V}_k\|_s\,ds.\label{eq:5.12}
	\end{align}
	By Gronwall's inequality (see, e.g., \citet[Lemma 4.1]{Pang_Talreja_Whitt_2007}) and (\ref{eq:5.12}), it follows that $\max_{k\in [K]}\|\bar{V}_k\|_T=0.$
	Since $T$ was arbitrary, it follows that $\bar{V}\equiv 0$.
    \end{proof}

\begin{corollary}\label{Cor:5.3}
	As $n\rightarrow\infty$, $T^n\Rightarrow e\in C^J$, where $e(t) \coloneqq  (t,\dots, t)$ \,for\, $t\ge 0$.
\end{corollary}
\begin{proof} By the definition in (\ref{eq:2.15}), $T^n = e - I^n$. The result then follows by Lemma \ref{Lemma:5.2} since $I^n\Rightarrow \mathbf{0}$ as $n\rightarrow\infty$.\end{proof}

\subsection{Convergence of Diffusion Scaled Processes}\label{Section:5.2}

The next result establishes weak convergence of the diffusion-scaled ``primitive'' processes $\hat{\xi}^n$ and $\hat{\zeta}^n$:
\begin{lemma}\label{Lemma:5.4}
	As $n\rightarrow\infty$, $(\hat{\xi}^n, \hat{\zeta}^n)\Rightarrow (\xi^*, \zeta^*)$, where $(\xi^*, \zeta^*)$ is $(J+K)$-dimensional Brownian motion with initial state $(Q(0),V(0))$ and covariance matrix $\Sigma$ given by (\ref{eq:3.29*})--(\ref{eq:3.33*}).
\end{lemma}
\begin{proof}
By Lemma~\ref{Lemma:5.2}, Corollary~\ref{Cor:5.3}, Donsker's theorem, the functional central limit theorem for renewal processes, and the continuous mapping theorem, we get weak convergence as $n\rightarrow\infty$ for the following fluid-scaled processes $j\in [J]$ and $k\in [K]$: 
\begin{align*}
    \!\!\!\!\!\!\!\!\!\!\!\!\bar{M}_k^n(\eta_k\halfspace\cdot)\Rightarrow \eta_k e,\,\,\,\bar{N}_j^n(\mu_j\halfspace\cdot)\Rightarrow \mu_j e,\,\,\,T_j^n\Rightarrow e,\,\,\, \rttensortwo{V}_k^n\Rightarrow m_k,
\end{align*}
where $e:[0,\infty)\rightarrow [0,\infty)$ denote the one-dimensional identity map $e(t)=t$ for $t\ge 0$. Similarly, we get weak convergence as $n\rightarrow\infty$ for the following diffusion-scaled processes for $j\in [J]$ and $k\in [K]$:
\begin{align*}
    \!\!\!\!\!\!\!\!\!\!\!\!&\hat{\Psi}_{kj}^n\Rightarrow \sqrt{q_{kj}(1-q_{kj})}\,B_{kj},\,\,\, \hat{\Phi}_{kj}^n\Rightarrow \sqrt{p_{jk}(1-p_{jk})}\,\tilde{B}_{jk},\\
    \!\!\!\!\!\!\!\!\!\!\!\!&\hat{M}_k^n(\eta_k\halfspace\cdot)\,\Rightarrow \sqrt{\eta_k}\,B_k,\,\,\,\,\,\,\,\,\,\,\,\,\,\,\,\,\hat{N}_j^n(\mu_j\halfspace\cdot)\,\Rightarrow \sqrt{\mu_j}\,\tilde{B}_j,
\end{align*}
where $B_{kj}$, $\tilde{B}_{jk}$, $B_k$, and $\tilde{B}_j$ are independent standard Brownian motions. (These convergence arguments are routine and therefore omitted for brevity.) Furthermore, the mapping $H:D\rightarrow D$, defined by $H(x)(t) \coloneqq \int_{0}^{t}\!x(s)\,ds$ for $(x,t)\in D\times[0,\infty)$, is continuous in the Skorokhod $J_1$ topology (see, e.g., \citet[page 229]{Pang_Talreja_Whitt_2007}), which implies that 
	\begin{align*}
	H(\rttensortwo{V}_k^n)\,\Rightarrow\, H(m_k)\,=\,m_k e\quad \text{as}\quad n\rightarrow\infty.
	\end{align*}
	By the above weak convergence results, \citet[Theorems 11.4.4 and 11.4.5]{Whitt_2002}, Assumption \ref{Assumption_Initial_Conditions}, and the independence of the stochastic model primitives, it follows that the (joint) processes $(\hat{Q}^n(0),\,\hat{V}^n(0),\,\hat{\Psi}^n,\,\hat{\Phi}^n,\,
	\hat{N}^n,\dots,\, \hat{M}^n)$ and $(T^n,\,\rttensortwo{V}^n,\,\bar{N}^n,\,\bar{M}^n)$
	converge weakly as $n\rightarrow\infty$ to their appropriate limits. From this, (\ref{eq:3.18})--(\ref{eq:3.19}), the random time change theorem, and the continuous mapping theorem, it follows that $(\hat{\xi}^n,\hat{\zeta}^n)$ converges weakly as $n\rightarrow\infty$ to $(J+K)$-dimensional Brownian motion $(\xi^*,\zeta^*)$ with initial state $(Q(0), V(0))$ and covariance matrix $\Sigma$ given by (\ref{eq:3.29*})--(\ref{eq:3.33*}). Because it is straightforward, albeit tedious, to derive the entries of the covariance matrix $\Sigma$, we omit the details.
    \end{proof}

\begin{proof}[Proof of Theorem \ref{Theorem:Main_Result}]
	It is straightforward to show that the process $(\hat{\xi}^n, \hat{\zeta}^n)$ defined by (\ref{eq:3.18})--(\ref{eq:3.19}) satisfies 
	\begin{alignat}{3}
	&\!\!\!\!\!\!\!\hat{\xi}^n_j(0) \,=\, \hat{Q}^n_j(0)\,\ge\, 0,&&\quad\text{for all } j\in [J],\label{eq:5.21}\\
	&\!\!\!\!\!\!\!\sum_{j=1}^{J}\hat{\xi}_j^n(t) \,+ \sum_{k=1}^{K}\hat{\zeta}_k^n(t)\,=\,0,&&\quad\text{for all } t\ge 0.\label{eq:5.22}
	\end{alignat}
	Moreover, as in the proof of Lemma \ref{Lemma:5.2}, 
    it is straightforward to show that $\hat{I}^n$ is nondecreasing componentwise with $\hat{I}^n(0)=0$ and satisfies
	\begin{align}
	\int_{0}^{\infty}\!\mathds{1}_{\{\hat{Q}_j^n(t)\, >\, 0\}}\,d\hat{I}^n_j(t)\,=\,0\quad\text{for all}\quad j\in[J].\label{eq:Diffusion_Idleness_Increases}
	\end{align}
	Since $(\hat{\xi}^n, \hat{\zeta}^n)\in D^{J+K}$ pathwise, it follows from (\ref{eq:3.20})--(\ref{eq:3.21}), (\ref{eq:5.21})--(\ref{eq:Diffusion_Idleness_Increases}), and Proposition~\ref{Prop:4.1} that $(\hat{Q}^n, \hat{I}^n,\hat{V}^n) = f(\hat{\xi}^n, \hat{\zeta}^n)$. Then, by Proposition~\ref{Prop:4.3}, Lemma~\ref{Lemma:5.4}, and the continuous mapping theorem, it follows that as $n\rightarrow\infty$,
	\begin{align*}
	(\hat{Q}^n, \hat{I}^n,\hat{V}^n) \,&=\, \big(f_1(\hat{\xi}^n, \hat{\zeta}^n), f_2(\hat{\xi}^n, \hat{\zeta}^n), f_3(\hat{\xi}^n, \hat{\zeta}^n)\big)\\\,&\Rightarrow\,(Q^*, I^*, V^*),
	\end{align*}
	where $Q^*\coloneqq f_1({\xi}^*, {\zeta}^*)$, $I^*\coloneqq f_2({\xi}^*, {\zeta}^*)$, and $V^*\coloneqq f_3({\xi}^*, {\zeta}^*)$. Moreover, since inequalities are preserved under weak convergence, Lemma \ref{Lemma:5.4} and (\ref{eq:5.21})--(\ref{eq:5.22}) imply that 
    \begin{alignat}{3}
	&\!\!\!\!\!\!\!\xi^*_j(0) \,=\, Q_j(0)\,\ge\, 0,&&\quad\text{for all } j\in [J],\label{eq:5.25}\\
	&\!\!\!\!\!\!\!\sum_{j=1}^{J}\xi_j^*(t) \,+ \sum_{k=1}^{K}\zeta_k^*(t)\,=\,0,&&\quad\text{for all } t\ge 0.\label{eq:5.26}
	\end{alignat}
    Therefore, since $(\xi^*,\zeta^*)\in C^{J+K}$ pathwise by Lemma \ref{Lemma:5.4}, it follows from (\ref{eq:5.25})--(\ref{eq:5.26}), Corollary \ref{Cor:4.2}, and Proposition~\ref{Prop:Preserve_Continuity} that $(Q^*, I^*, V^*)$ satisfies (\ref{eq:3.27})--(\ref{eq:3.29}) and has continuous sample paths. This completes the proof.
    \end{proof}



\appendix

\section*{\textsc{Appendix}}
\addcontentsline{toc}{section}{Appendices} 

\section{Proof of Lemma \ref{Lemma:A.1}}\label{appendix:C}
We prove that for each $T>0$, there exists a unique $y\in D^K_T$ satisfying (\ref{eq:A.1}) for all $t\in [0,T]$, then extend this solution to $D^K$ in the obvious way. To improve the readability of the argument, we break the proof into a few separate steps, organized as subsections. 

\subsection{Existence of an element in $D_K^T$ satisfying (\ref{eq:A.1}) for all $t\in [0,T]$} We prove existence via the method of successive approximations; see, e.g., \citet{Reed_Ward_2004} for a similar proof approach. In particular, we construct a sequence that is Cauchy in $D_K^T$ under the sup norm, and then argue that the limit of the sequence (which exists by completeness of $D_K^T$) satisfies (\ref{eq:A.1}).

Let $y_k^0\equiv 0$, and define $y_k^n\in D$ for $n\in \mathbb{N}$ iteratively as follows for each $k\in [K]$:
\begin{align}
\!\!\!\!\!\!\!y_k^n \coloneqq\,  \xi_k \,&-\, \eta_k\!\int_{0}^{\cdot}\!y_k^{n-1}(s)\,ds\notag\\\,&-\sum_{j=1}^{J}p_{jk}\psi\Big(\xi_j + \sum_{l=1}^{K}q_{lj}\eta_l\!\int_{0}^{\cdot}\!y_l^{n-1}\,ds\Big).\label{eq:B.1}
\end{align}
Then, the sequence $\{(y_1^n\vert_{[0,T]},\dots, y_K^n\vert_{[0,T]}):n=0,1,\dots \}$ of elements in $D_T^K$ defined by (\ref{eq:B.1}) is a Cauchy sequence with respect to the sup norm; see Lemma~\ref{Claim:B.1} at the end of this subsection for a proof of this claim. Therefore, by completeness of $(D_T^K, \|\cdot\|_{T,K})$, it follows that as $n\rightarrow\infty$, 
\begin{align}
\!\!\!\!\!\!\!\!\!\!\!\big(y_1^n\vert_{[0,T]},\dots, y_K^n\vert_{[0,T]}\big)\,\rightarrow\, \big(y^{\infty}_{1,T},\dots, y^{\infty}_{K,T}\big)\,\in\, D^K_T.\label{eq:B.2}
\end{align}
We claim that $(y^{\infty}_{1,T},\dots, y^{\infty}_{K,T})$ satisfies (\ref{eq:A.1}) for all $t\in [0,T]$. To show this, consider the mapping $L:D^K_T\rightarrow D^K_T$ defined as follows:\newline\vspace{-1.25em}
\begin{align*}
&\!\!\!\!\!\!\!\!(y_1,\dots, y_K)\notag\\&\!\!\!\!\!\!\!\!\qquad\mapsto\,\Big(\zeta_k\,-\,\eta_k\!\int_{0}^{\cdot}\!y_k(s)\,ds\notag\\&\!\!\!\!\!\!\!\!\qquad\qquad-\sum_{j=1}^{J}p_{jk}\psi\Big(\xi_j +\sum_{l=1}^{K}q_{lj}\eta_l\!\int_{0}^{\cdot}\!y_l(s)\,ds\Big)\Big)_{k\in [K]}.
\end{align*}
Then, for $y, \tilde{y}\in D_T^K$, it follows by the definition of $L$, the triangle inequality, and \citet[Lemma 13.5.1]{Whitt_2002} that
\begin{align*}
&\!\!\!\!\!\!\!\!\!\|L(y)- L(\tilde{y})\|_{T,K}\\&\!\!\!\!\!\!\!\!\!\quad\le\, \max_{k\in [K]}\Big\{\eta T\|y_k - \tilde{y}_k\|_T +\sum_{j=1}^{J}\sum_{l=1}^{K}\eta T\|y_l-\tilde{y}_l\|_T\Big\}\\
&\!\!\!\!\!\!\!\!\!\quad\le\, \eta T (JK+1) \|y-\tilde{y}\|_{T,K},
\end{align*}
implying that $L$ is Lipschitz continuous. It then follows from (\ref{eq:B.1})--(\ref{eq:B.2}) that
\begin{align*}
\big(y^{\infty}_{1,T},\dots, y^{\infty}_{K,T}\big)\,&\leftarrow\, \big(y_1^{n+1}\vert_{[0,T]},\dots, y_K^{n+1}\vert_{[0,T]}\big)\\\,&=\,L\big(y_1^n\vert_{[0,T]},\dots, y_K^n\vert_{[0,T]}\big)\\\,&\rightarrow\, L\big(y^{\infty}_{1,T},\dots, y^{\infty}_{K,T}\big),
\end{align*}
as $n\rightarrow\infty$. By uniqueness of limits in metric spaces, it follows that $L(y^{\infty}_{1,T},\dots, y^{\infty}_{K,T})= (y^{\infty}_{1,T},\dots, y^{\infty}_{K,T})$, implying that $(y^{\infty}_{1,T},\dots, y^{\infty}_{K,T})$ satisfies (\ref{eq:A.1}) for all $t\in [0,T]$.

We conclude this subsection with a proof showing that the sequence defined by (\ref{eq:B.1}) is Cauchy:
\begin{lemma}\label{Claim:B.1}
	For each $T>0$, the sequence $$\big\{\big(y_1^n\vert_{[0,T]},\dots, y_K^n\vert_{[0,T]}\big):n=0,1,\dots\! \big\}$$ defined by (\ref{eq:B.1}) is Cauchy in $D_K^T$ with respect to $\|\cdot\|_{T,K}$.
\end{lemma}
\begin{proof}Fix $\delta\in (0, T)$ such that $2\delta\eta JK<1$. (This choice of $\delta$ will be used crucially later.) First, we claim that 
\begin{align}
	\|y_k^n - y_k^{n-1}\|_{\delta} \,\le\,(2\delta \eta JK)^{n-1}C_\delta,\label{eq:B.21}
\end{align}
for all $n\in \mathbb{N}$ and $k\in [K]$, where $C_t:=\max_{k\in [K]}\|\zeta_k\|_t + \sum_{j=1}^{J}\|\psi(\xi_j)\|_t$ for $t>0$. (Note that $C_{t_1}\le C_{t_2}$ for all $0\le t_1\le t_2<\infty$.) To prove (\ref{eq:B.21}), note that for $n=1$,
	\begin{align*}
	\!\!\!\|y_k^1 - y_k^0\|_\delta \,&=\, \big\|\zeta_k - \sum_{j=1}^{J}p_{jk}\psi(\xi_j)\big\|_\delta \notag\\[-0.5em]&\le\, \|\zeta_k\|_\delta  + \sum_{j=1}^{J}\|\psi(\xi_j)\|_\delta\,\le\, C_\delta,
	\end{align*}
	for all $k\in [K]$. Moreover, note that for $n\ge 2$,
	\begin{align}
	&\!\!\!\!\!\!\!\!\|y_k^n - y_k^{n-1}\|_{\delta}\notag\\[0.75em]
	&\!\le\, \eta_k \delta\|y_k^{n-1}-y_{k}^{n-2}\|_{\delta} \notag\\[0.25em]&\,\,\,\,\,\,\,\,\,\,\,+ \sum_{j=1}^{J}p_{jk}\Big\|\sum_{l=1}^{K}q_{lj}\eta_l\!\int_{0}^{\cdot}\!\big(y_l^{n-1}(s)- y_l^{n-2}(s)\big)\,ds\Big\|_{\delta}\notag\\
	&\!\le\, 2\delta \eta J\sum_{l=1}^{K}\|y_l^{n-1}-y_l^{n-2}\|_{\delta}\notag,
	\end{align}
	for all $k\in [K]$, where the first inequality follows from \citet[Lemma 13.5.1]{Whitt_2002} and (\ref{eq:B.1}). By performing a similar estimate for $\|y_l^{n-1}-y_l^{n-2}\|_{\delta}$ in the above display, it follows that
	\begin{align*}
	\!\!\!\!\|y_k^n - y_k^{n-1}\|_{\delta} 
    \,\le\, (2\delta \eta J)^2K\sum_{l=1}^{K}\|y_l^{n-2}-y_l^{n-3}\|_{\delta},
	\end{align*}
	for all $k\in [K]$. Continuing in this way, we obtain (\ref{eq:B.21}). Next, we claim that for each $k\in [K]$, 
	\begin{align}
	\|y_k^n- y_k^{n-1}\|_{m\delta}\,\le\, m^2n^m (2\delta \eta J K)^{n-1}C_{m\delta},\label{eq:B.22}
	\end{align}
	for all $m,n\in\mathbb{N}$. Fixing $k\in [K]$, we consider the $n=1$ and $n\ge 2$ cases separately. When $n=1$, we have that $\|y_k^1- y_k^{0}\|_{m\delta}\le C_{m\delta}\le m^2 C_{m\delta}$ for all $m\in\mathbb{N}$ by (\ref{eq:B.21}), so that (\ref{eq:B.22}) holds. When $n\ge 2$, we proceed by (strong) induction on $m$. The base case of $m=1$ holds immediately by (\ref{eq:B.21}). 
    For the inductive step, we assume that for all $n\ge 2$,
	\begin{align}
	\|y_k^n- y_k^{n-1}\|_{r\delta}\,\le\, r^2n^r (2\delta \eta J K)^{n-1}C_{r\delta},
    \label{eq:B.24}
	\end{align}
	for all $r\in [m]$. Then, by (\ref{eq:B.24}), it follows that for $n\ge 2$,
	\begin{align}
	&\!\!\!\!\!\!\!\!\!\!\!\!\|y_k^n - y_k^{n-1}\|_{(m+1)\delta}\notag\\
	&\!\!\!\!\!\!\!\!\le\, \eta_k \sum_{r=1}^{m+1}\int_{(r-1)\delta}^{r\delta}\|y_k^{n-1}-y_k^{n-2}\|_{r\delta}\,ds\notag\\
	&+ \sum_{j=1}^{J}p_{jk}\Big\|\sum_{l=1}^{K}q_{lj}\eta_l\!\int_{0}^{\cdot}\!\vert y_l^{n-1}(s) - y_l^{n-2}(s)\vert\,ds\Big\|_{(m+1)\delta}\notag\\
	&\!\!\!\!\!\!\!\!=\,\delta \eta\Big[\sum_{r=1}^{m+1}\|y_k^{n-1}- y_k^{n-2}\|_{r\delta} \notag\\&\qquad+\, J\sum_{l=1}^{K}\sum_{r=1}^{m+1}\|y_l^{n-1}-y_l^{n-2}\|_{r\delta}\Big]\notag\\
	&\!\!\!\!\!\!\!\!\le\, 2\delta \eta J\sum_{l=1}^{K}\sum_{r=1}^{m}r(n-1)^r(2\delta\eta JK)^{n-2}C_{m\delta} \notag\\&\qquad +\, 2\delta\eta J\sum_{l=1}^{K}\|y_l^{n-1}-y_l^{n-2}\|_{(m+1)\delta}\notag\\
	&\!\!\!\!\!\!\!\!=\,(2\delta \eta JK)^{n-1}C_{r\delta}\sum_{r=1}^{m}r(n-1)^r\notag\\&\qquad+\, 2\delta\eta J\sum_{l=1}^{K}\|y_l^{n-1}-y_l^{n-2}\|_{(m+1)\delta}\notag\\
	&\!\!\!\!\!\!\!\!\le\, (2\delta \eta JK)^{n-1}C_{m\delta} m^2n^m \notag\\&\qquad +\, 2\delta\eta J\sum_{l=1}^{K}\|y_l^{n-1}-y_l^{n-2}\|_{(m+1)\delta}.\notag
	\end{align}
	Continuing from the above display, it follows from (\ref{eq:B.21}) that for $n=2$,
	\begin{align*}
	&\!\!\!\!\!\!\!\|y_k^2 - y_k^1\|_{(m+1)\delta}\notag\\&\le\, (2\delta\eta JK)C_{m\delta} m^2 2^m + 2\delta\eta J\sum_{l=1}^{K}\|y_l^1-y_l^0\|_{(m+1)\delta}\notag\\
    &\le\, (2\delta\eta JK)C_{(m+1)\delta}(m^2 2^m + 1).
	\end{align*}
	Continuing in this way, it follows that for $n\ge 2$,
	\begin{align}
	&\!\!\!\!\!\!\!\|y_k^n - y_k^{n-1}\|_{(m+1)\delta}\notag\\&\le\, (2\delta \eta JK)^{n-1}C_{(m+1)\delta}\big(m^2\sum_{i=2}^{n}i^m +1\big)\notag\\
	&\le\, (m+1)n^{m+1}(2\delta \eta JK)^{n-1}C_{(m+1)\delta}.\notag
	\end{align}
	This completes the inductive step. This proves that (\ref{eq:B.22}) holds for all $m, n\in\mathbb{N}$.
	
	Equipped with (\ref{eq:B.22}), it now follows that
	\begin{align}
	&\!\!\!\!\!\!\!\|y_k^n - y_k^{n-1}\|_T\notag\\&\le\, \|y_k^n - y_k^{n-1}\|_{\lceil \delta^{-1}T\rceil \delta}\notag\\&\le\, \lceil \delta^{-1}T\rceil^2 n^{\lceil \delta^{-1}T\rceil} (2\delta JK)^{n-1}C_{\lceil \delta^{-1}T\rceil \delta},\notag
	\end{align}
	for all $n\in\mathbb{N}$ and $k\in [K]$, implying that
	\begin{align}
	&\!\!\!\!\!\!\!\|y_k^n - y_k^{n-1}\|_{T,K}\notag\\&\le\, \lceil \delta^{-1}T\rceil^2 n^{\lceil \delta^{-1}T\rceil} (2\delta JK)^{n-1}C_{\lceil \delta^{-1}T\rceil \delta}.\label{eq:B.30}
	\end{align}
	Thus, to prove that the sequence $\{(y_1^n\vert_{[0,T]},\dots, y_K^n\vert_{[0,T]}):n=0,1,\dots \}$ defined by (\ref{eq:B.1}) is Cauchy in $D_T^K$, it suffices to show that the right-hand side of (\ref{eq:B.30}) converges to zero as $n\rightarrow\infty$. But, by our choice of $\delta$, note that
	\begin{align*}
	&\!\!\!\!\!\!\!\limsup_{n\rightarrow\infty} \Bigg\vert\frac{\lceil \delta^{-1}T\rceil^2 (n+1)^{\lceil \delta^{-1}T\rceil} (2\delta JK)^{n}C_{\lceil \delta^{-1}T\rceil \delta}}{\lceil \delta^{-1}T\rceil^2 n^{\lceil \delta^{-1}T\rceil} (2\delta JK)^{n-1}C_{\lceil \delta^{-1}T\rceil \delta}}\Bigg\vert \\&=\, \limsup_{n\rightarrow\infty}\Big[\frac{(n+1)^{\lceil \delta^{-1}T\rceil}\,2\delta JK}{n^{\lceil \delta^{-1}T\rceil}} \Big]\,=\,  2\delta JK\,<\,1,
	\end{align*}
	implying that $\sum_{n=1}^{\infty}\lceil \delta^{-1}T\rceil^2 n^{\lceil \delta^{-1}T\rceil} (2\delta JK)^{n-1}C_T<\infty$ by the ratio test. Hence, the right-hand side of (\ref{eq:B.30}) converges to zero as $n\rightarrow\infty$, completing the proof.\end{proof}

\subsection{Uniqueness of the element in $D_K^T$ satisfying (\ref{eq:A.1}) for all $t\in [0,T]$} We show that $(y^{\infty}_{1,T},\dots, y^{\infty}_{K,T})\in D_T^K$ given by (\ref{eq:B.2}) is the unique element in $D_T^K$ satisfying (\ref{eq:A.1}) for all $t\in [0,T]$. Suppose that $(y_1,\dots, y_K), (\tilde{y}_1,\dots, \tilde{y}_K)\in D^K_T$ both satisfy (\ref{eq:A.1}) for all $t\in [0,T]$. We partition the interval $[0,T]$ into a finite number of subintervals and show that both solutions agree on each of these subintervals. To that end, define
\begin{align*}
m \coloneqq  \inf\big\{n\ge 1: \frac{n}{2}\left(2JK\eta\right)^{-1}>T\big\}.
\end{align*}
First, letting $t_1\coloneqq \tfrac{1}{2}\left(2JK\eta\right)^{-1}$, we have that
\begin{align}
&\!\!\!\!\!\!\!\max_{k\in [K]}\|y_k-\tilde{y}_k\|_{t_1}\notag\\ 
&\le\, \max_{k\in [K]}\eta\!\int_{0}^{t_1}\!\|y_k -\tilde{y}_k\|_{t_1}\,ds \notag\\&\qquad\qquad+ \sum_{j=1}^{J}\sum_{l=1}^{K}\eta\!\int_{0}^{t_1}\!\|y_l - \tilde{y}_l\|_{t_1}\,ds\notag\\
&\le\, 2JK\eta \, t_1\max_{k\in [K]}\|y_k -\tilde{y}_k\|_{t_1}\notag\\&=\,\frac{1}{2}\max_{k\in [K]}\|y_k -\tilde{y}_k\|_{t_1}.\notag
\end{align}
Hence, it follows that $\max_{k\in [K]}\|y_k -\tilde{y}_k\|_{t_1}=0$, implying that $y \equiv \tilde{y}$ on $[0, t_1]$. Second, letting $t_2\coloneqq (2JK\eta)^{-1}$, we have that
\begin{align}
&\!\!\!\!\!\!\!\max_{k\in [K]}\|y_k - \tilde{y}_k\|_{t_2}\notag\\
&=\, \max_{k\in [K]}\eta\, \Big\|\int_{0}^{t_1}\vert y_k(s)-\tilde{y}_k(s)\vert\,ds \notag\\&\qquad\qquad\qquad + \int_{t_1}^{\cdot}\vert y_k(s)- \tilde{y}_k(s)\vert\,ds\Big\|_{t_2}\notag\\&\qquad\qquad +\,\eta J\sum_{l=1}^{K}\Big\|\int_{0}^{t_1}\vert y_l(s)-\tilde{y}_l(s)\vert\,ds \notag\\&\qquad\qquad\qquad + \int_{t_1}^{\cdot}\vert y_l(s)- \tilde{y}_l(s)\vert\,ds\Big\|_{t_2}\notag\\
&\le\, \max_{k\in [K]}\big\{\eta\, t_1\|y_k-\tilde{y}_k\|_{t_1} \,+\, (t_2-t_1)\|y_k-\tilde{y}_k\|_{t_2}\big\}\notag\\&\qquad +\,\eta J\sum_{l=1}^{K}\big[\,t_1\|y_l-\tilde{y}_l\|_{t_1} + (t_2-t_1)\|y_l-\tilde{y}_l\|_{t_2}\,\big]\notag\\
&\le\, 2\eta JK(t_2-t_1)\max_{k\in [K]}\|y_k-\tilde{y}_k\|_{t_2}\notag\\
&=\,\frac{1}{2}\max_{k\in [K]}\|y_k -\tilde{y}_k\|_{t_2},\notag
\end{align}
where the third inequality follows from the fact that $y\equiv \tilde{y}$ on $[0,t_1]$. Hence, it follows that $\|y_k-\tilde{y}_k\|_{t_2}=0$, implying that $y \equiv \tilde{y}$ on $[0, t_2]$. 

Continuing in an iterative fashion, the same argument shows that $y\equiv \tilde{y}$ on $[0,t_n]$ for all $n\in [m-1]$, where $t_n \coloneqq \tfrac{n}{2}(2JK\eta)^{-1}$ for $n\in [m-1]$. 
If $t_{m-1}=T$, then we are done. However, if $t_{m-1}<T$, then we can let $t_m \coloneqq T$ and use the same argument to show that $y\equiv\tilde{y}$ on $[0,t_m]$. 

\subsection{Extension to a unique element in $D^K$ satisfying (\ref{eq:A.1}) for all $t\in [0,\infty)$} 
The previous two subsections have shown that, for each $T>0$, there exists a unique $(y^{\infty}_{1,T},\dots, y^{\infty}_{K,T})\in D_T^K$ satisfying (\ref{eq:A.1}) for all $t\in [0,T]$. Using these solutions, we construct an element in $D^K$ that uniquely satisfies (\ref{eq:A.1}) for all $t\in [0,\infty)$. To that end, define $(y^{\infty}_{1},\dots, y^{\infty}_{K})\in D^K$ by 
\begin{align*}
\!\!\!\!\!\!\!\!\!\!\big(y^{\infty}_{1},\dots, y^{\infty}_{K}\big)(t)\coloneqq\,  \big(y^{\infty}_{1, T},\dots, y^{\infty}_{K, T}\big)(t),\quad t\in [0,T],
\end{align*}
for each $T>0$. We claim that $(y^{\infty}_{1},\dots, y^{\infty}_{K})$ is well-defined and is the unique element in $D^K$ satisfying (\ref{eq:A.1}) for all $t\in [0,\infty)$. To prove that it is well-defined, we must show that whenever $t\in [0,T_1]\cap [0,T_2]$,
\begin{align}
\big(y_{1, T_1}^{\infty},\dots, y_{K, T_1}^{\infty}\big)(t) \,=\, \big(y_{1, T_2}^{\infty},\dots, y_{K, T_2}^{\infty}\big)(t).\label{eq:B.16}
\end{align}
Without loss of generality, suppose that $T_1\le T_2$. Then $(y_{1, T_2}^{\infty},\dots, y_{K, T_2}^{\infty})\vert_{[0,T_1]}\in D^K_{T_1}$ satisfies (\ref{eq:A.1}) for all $t\le T_1$. By uniqueness, 
\begin{align*}
\big(y_{1, T_1}^{\infty},\dots, y_{K, T_1}^{\infty}\big)\,=\,\big(y_{1, T_2}^{\infty},\dots, y_{K, T_2}^{\infty}\big)\big\vert_{[0,T_1]},
\end{align*}
which implies that
\begin{align*}
\big(y_{1, T_1}^{\infty},\dots, y_{K, T_1}^{\infty}\big)(t) \,&=\, \big(y_{1, T_2}^{\infty},\dots, y_{K, T_2}^{\infty}\big)\big\vert_{[0,T_1]}(t) \\\,&=\, \big(y_{1, T_2}^{\infty},\dots, y_{K, T_2}^{\infty}\big)(t),
\end{align*}
for all $t\in [0,T_1]$, proving (\ref{eq:B.16}) holds. Finally, by the construction, it is obvious that $(y^{\infty}_{1},\dots, y^{\infty}_{K})$ uniquely satisfies (\ref{eq:A.1}) for all $t\in [0,\infty)$.\hfill\qedsymbol


\section{Proof of Proposition \ref{Prop:4.1}}\label{appendix:A}
Fix $(\xi, \zeta)\in D^{J+K}$ satisfying (\ref{eq:4.1})--(\ref{eq:4.2}). We first prove existence. By Lemma \ref{Lemma:A.1}, there exists a $y\in D^K$ satisfying (\ref{eq:A.1}). Then, for $j\in [J]$, define
	\begin{alignat}{2}
	\!\!\!\!\!\!\!\!\!\!\!u_j(t)&\coloneqq\,  \mu_j^{-1}\psi\Big(\xi_j \,+ \sum_{l=1}^{K}q_{lj}\eta_l\!\int_{0}^{t}\!y_l(s)\,ds\Big),&&\,\,\,\, t\ge 0,\!\!\!\!\!\label{eq:A.2}\\
	\!\!\!\!\!\!\!\!\!\!\!x_j(t)&\coloneqq\,  \phi\Big(\xi_j \,+ \sum_{l=1}^{K}q_{lj}\eta_l\!\int_{0}^{t}\!y_l(s)\,ds\Big),&&\,\,\,\, t\ge 0.\!\!\!\!\!\label{eq:A.3}
	\end{alignat}
	Since $y\in D^K$, it follows that $u\in D^J$ and $x\in D^J$, so that $(x,u,y)\in D^{2J + K}$. To prove existence, it suffices to show that $(x,u,y)$ satisfies (\ref{eq:4.3})--(\ref{eq:4.7}). Equation (\ref{eq:4.3}) follows from (\ref{eq:D.2})--(\ref{eq:D.3}) and (\ref{eq:A.2})--(\ref{eq:A.3}). Equation (\ref{eq:4.4}) follows from (\ref{eq:A.1})--(\ref{eq:A.2}). Equation (\ref{eq:4.5}) follows from (\ref{eq:4.1}), (\ref{eq:4.3})--(\ref{eq:4.4}), and the fact that the matrices in (\ref{Equation:Stochastic_Matrix}) are stochastic. Equation (\ref{eq:4.6}) follows from (\ref{eq:D.2}), (\ref{eq:4.1}), and (\ref{eq:A.2}). Finally, for $j\in [J]$, define 
	\begin{align}
	z_j(t)\coloneqq\,  \xi_j \,+ \sum_{l=1}^{K}q_{lj}\eta_l\!\int_{0}^{t}\!\!y_l(s)\,ds,\quad t\ge 0.\label{eq:A.4}
	\end{align} 
	Thus, $u_j=\mu_j^{-1}\psi(z_j)$ and $x_j = \phi(z_j)$. But, by (\ref{eq:D.2})--(\ref{eq:D.3}), it follows that $\int_{0}^{\infty}\mathds{1}_{\{ x_j>0\}}\,d(u_j)(t)=0.$ Therefore, (\ref{eq:4.7}) holds.
	
	We now prove uniqueness. Suppose $(x, u, y), (\tilde{x}, \tilde{u}, \tilde{y})\in D^{2J+K}$ both satisfy (\ref{eq:4.3})--(\ref{eq:4.7}). Then, by (\ref{eq:4.3}), we have that
	\begin{align}
	\!x_j \,=\, z_j \,+\, \mu_j u_j\,\ge\, 0,\quad \tilde{x}_j \,=\, \tilde{z}_j \,+\,\mu_j \tilde{u}_j\,\ge\, 0,\label{eq:A.7}
	\end{align}
	for $j\in [J]$, where $z_j$ and $\tilde{z}_j$ are given by (\ref{eq:A.4}), with $y_j$ and $\tilde{y}_j$, respectively. Since $(x,u)$ and $(\tilde{x},\tilde{u})$ both satisfy (\ref{eq:4.6})--(\ref{eq:4.7}), it follows that $(x,\mu_j u_j)$ and $(\tilde{x},\mu_j \tilde{u}_j)$ also both satisfy (\ref{eq:4.6})--(\ref{eq:4.7}). By this and (\ref{eq:A.7}), for $j\in [J]$ we have that
	\begin{align}
	\,\,\,\,\,\,\,\,\,\mu_j u_j \,=\, \psi(z_j)\quad\text{and}\quad \mu_j \tilde{u}_j \,=\, \psi(\tilde{z}_j).\label{eq:A.9}
	\end{align}
	It then follows from (\ref{eq:4.4}), (\ref{eq:A.4}), (\ref{eq:A.9}), and Lemma \ref{Lemma:A.1} that $y_k = \tilde{y}_k$ for all $k\in [K]$. By uniqueness of $y$, it follows from (\ref{eq:A.4}) and (\ref{eq:A.9}) that $u_j = \tilde{u}_j$ for all $j\in [J]$. Finally, by uniqueness of $y$ and $u$, it follows from (\ref{eq:A.4})--(\ref{eq:A.7}) that $x_j = \tilde{x}_j$ for all $j\in [J]$. This completes the proof.

\section{Proof of Proposition \ref{Prop:4.3}}
\label{appendix:B}

It is easy to show that the function $f=(f_1,f_2,f_3)$ is continuous if the functions $f_1$, $f_2$, and $f_3$ are continuous. On the other hand, it is straightforward to prove continuity of $f_1$ and $f_2$ once the continuity of $f_3$ is established. Thus, we only prove continuity of $f_3$ and omit the continuity proofs of $f_1$, $f_2$, and $f$ for the purposes of brevity.

In the Skorokhod $J_1$ topology (see, e.g., \citet{Billingsley_1999} and \citet{Whitt_2002}), a sequence $\{x^n\}_{n=1}^{\infty}$ in $D^k$ converges to an element $x\in D^k$, written $x^n\rightarrow x$, as $n\rightarrow\infty$, if $d_{T}^k(x^n\vert_{[0,T]}, x\vert_{[0,T]})\rightarrow 0$ as $n\rightarrow\infty$ for all continuity points $T>0$ of $x$,
where $d_T^k:D_T^k\,\times\, D_T^k\rightarrow [0,\infty)$ is given by 
\begin{align*}
\!\!\!\!\!d_T^k(x,y)\coloneqq \inf_{\lambda\in \Lambda_T}\!\big\{\|x\,\circ\, \lambda\,-\,y\|_{T,k}\,\vee\, \|\lambda \,-\, e\|_T\big\},
\end{align*}
where $e:[0,T]\rightarrow [0,T]$ is the identity map
, and 
\begin{align}
\Lambda_T\coloneqq \big\{\lambda&:[0,T]\rightarrow[0,T]\nonumber\\&\vert\,\lambda\text{ is an increasing homeomorphism}\big\}.\notag
\end{align}
To that end, suppose that $(\xi^n,\zeta^n)\rightarrow (\xi,\zeta)$ in $D^{J+K}$ as $n\rightarrow\infty$ and let $\tilde{T}>0$ be a continuity point of $f_3(\xi, \zeta)$. To prove that $f_3$ is continuous with respect to the Skorokhod $J_1$ topology, we must show that 
\begin{align}
\lim_{n\rightarrow\infty} d_{\tilde{T}}^K\big(f_3(\xi^n,\zeta^n)\vert_{[0,\tilde{T}]}, f_3(\xi, \zeta)\vert_{[0,\tilde{T}]}\big)\,=\,0.\notag
\end{align}
Since $(\xi, \zeta)\in D^{J+K}$, it has at most countably many points of discontinuity; see, e.g., \citet[Lemma 5.1]{Ethier_Kurtz_2005}. It follows that there exists some $T>\tilde{T}$ that is a continuity point of $(\xi,\zeta)$. Therefore, by \citet[Lemma 1, page 167]{Billingsley_1999}, it suffices to show that
	\begin{align}
	\lim_{n\rightarrow\infty}d_{T}^K\big(f_3(\xi^n,\zeta^n)\vert_{[0,T]}, f_3(\xi, \zeta)\vert_{[0,T]}\big)\,=\,0.\label{eq:A.15}
	\end{align}
    However, we note that (\ref{eq:A.15}) can equivalently be written with $\tilde{d}_T^{\,\,k}$ in place of $d_T^k$ (see, e.g., page 226 of \citet{Pang_Talreja_Whitt_2007} and \citet{Billingsley_1999}), where 
	$\tilde{d}_T^{k}:D_T^k\times D_T^k\rightarrow [0,\infty)$ is given by 
	\begin{align*}
	\tilde{d}_T^{k}(x,y)\coloneqq \inf_{\lambda\in \tilde{\Lambda}_T}\!\big\{\|x\,\circ\, \lambda-y\|_{T,k}\vee\|\dot{\lambda}-1\|_T\big\},
	\end{align*}
	where $\dot{\lambda}$ is the derivative of $\lambda$, $1$ is the constant function taking the value one everywhere, and
	\begin{align*}
	\!\!\!\!\!\!\tilde{\Lambda}_T\coloneqq\big\{\lambda\in\Lambda_T:\lambda&\text{ is absolutely continuous}\notag\\&\qquad \text{w.r.t. Lebesgue measure}\big\}.
	\end{align*} 
	Therefore, the remainder of the proof aims at proving (\ref{eq:A.15}) with $\tilde{d}_T^{K}$ in place of $d_{T}^K$. To avoid overly cumbersome notation, we write $f_3(\xi^n,\zeta^n)$ and $f_3(\xi, \zeta)$ to mean $f_3(\xi^n,\zeta^n)\vert_{[0,T]}$ and $f_3(\xi, \zeta)\vert_{[0,T]}$, respectively.
    
    Since $T>0$ is a continuity point of $(\xi, \zeta)$ and $(\xi^n,\zeta^n)\rightarrow (\xi,\zeta)$ in $D^{J+K}$ as $n\rightarrow\infty$, there exists a sequence of homeomorphisms $\lambda^n\in\tilde{\Lambda}_T$ such that 
	\begin{align}
	\!\!\!\!\!\!\!\!\!\!\|(\xi,\zeta)\,\circ\, \lambda^n \,-\, (\xi^n,\zeta^n)\|_{T, J+K}\,\vee\, \|\dot{\lambda}^n\,-\,1\|_{T}\,\rightarrow\, 0,\label{eq:A.17}
	\end{align}
	as $n\rightarrow\infty$. Then, letting $y:= f_3(\xi,\zeta)$ and $y^n:= f_3(\xi^n, \zeta^n)$, it follows from the definition of $f_3$ and the triangle inequality that for any $0\le t\le T$,
	\begin{align}
	&\hspace{-2.5em}\max_{k\in [K]}\|y_k\,\circ\, \lambda^n \,-\, y^n_k\|_t\notag\\
	&\hspace{-2.5em}\,\,\,\,\,\,\le\,\max_{k\in [K]}\Big[\|\zeta_k\,\circ\,\lambda^n \,-\, \zeta_k^n\|_t\notag\\&\,\,\,+\,\eta\,\Big\|\int_{0}^{\lambda^n(\cdot)}\!\!\!y_k(s)\,ds\,-\int_{0}^{\cdot}\!y_k^n(s)\,ds\Big\|_t\notag\\
	&\,\,\,+\sum_{j=1}^{J}\Big\| \psi\Big(\xi_j \,+ \sum_{l=1}^{K}q_{lj}\eta_j\!\int_{0}^{\cdot}\!y_l(s)\,ds\Big)\circ \lambda^n\notag\\&\,\,\,\,\,\,\,\,\,\,\,\,\,\,\,\,\,\,\,\,\,-\,\psi\Big(\xi_j^n\,+\sum_{l=1}^{K}q_{lj}\eta_l\!\int_{0}^{\cdot}\!y_l^n(s)\,ds\Big)\Big\|_t\Big].\label{eq:A.18}
	\end{align}
	We next bound each term on the right-hand side of (\ref{eq:A.18}). First, let $M_T\coloneqq \max_{k\in [K]}\|y_k\|_T<\infty$. Then, by the chain rule, it follows that
	\begin{align}
	&\hspace{-1.5em}\Big\|\int_{0}^{\lambda^n(\cdot)}\!\!\!y_k(s)\,ds\,-\int_{0}^{\cdot}\!y_k^n(s)\,ds\Big\|_t\notag\\&\!=\,\Big\|\int_{0}^{\cdot}\!y_k(\lambda^n(s))\,\dot{\lambda}^n(s)\,ds\,-\int_{0}^{\cdot}\!y_k^n(s)\,ds\Big\|_t\notag\\
	&\!\le\,  \Big\| \int_{0}^{\cdot}\!y_k(\lambda^n(s))\big(\dot{\lambda}^n(s)\,-\,1\big)\,ds\Big\|_t\notag \\&\,\,\,\,\,\,\,\,\,\,\,+\, \Big\| \int_{0}^{\cdot}\!\big(y_k(\lambda^n(s))\,-\,y_k^n(s)\big)\,ds\Big\|_t\notag\\
	&\!\le\, TM_T\|\dot{\lambda}\,-\,1\|_T \,+ \int_{0}^{t}\|y_k\,\circ\, \lambda^n \,-\, y_k^n\|_s\,ds.\label{eq:A.20}
	\end{align}
	Similarly, by the chain rule, it follows that
	\begin{align}
	&\hspace{-1.25em}\Big\| \psi\Big(\xi_j \,+ \sum_{l=1}^{K}q_{lj}\eta_j\!\int_{0}^{\cdot}\!y_l(s)\,ds\Big)\,\circ\, \lambda^n\notag\\&\,\,\,\,\,\,\,\,\,\,\,\,-\,\psi\Big(\xi_j^n\,+\sum_{l=1}^{K}q_{lj}\eta_l\!\int_{0}^{\cdot}\!y_l^n(s)\,ds\Big)\Big\|_t\notag\\&\hspace{-0.5em}\le\, \|\xi_j\,\circ\,\lambda^n \,-\, \xi_j^n\|_T\notag\\&\,\,\,\,\,\,\,\,\,\,\,\,+\,\eta\sum_{l=1}^{K}\Big\|\int_{0}^{\lambda^n(\cdot)}\!\!\!y_l(s)\,ds \,- \int_{0}^{\cdot}\!y_l^n(s)\,ds\Big\|_t\notag\\
	&\hspace{-0.5em}\le\,\|\xi_j\,\circ\,\lambda^n \,-\, \xi_j^n\|_T\notag\\&\,\,\,\,\,\,\,\,\,\,\,\,+\,\eta\sum_{l=1}^{K}\Big\|\int_{0}^{\cdot}\!y_l(\lambda^n(s))\,\dot{\lambda}^n(s)\,ds \,- \int_{0}^{\cdot}\!y_l^n(s)\,ds\Big\|_t\notag\\&\hspace{-0.5em}\le\, \|\xi_j\,\circ\,\lambda^n \,-\, \xi_j^n\|_T\,+\,\eta KTM_T\|\dot{\lambda}^n\,-\,1\|_T\notag\\&\,\,\,\,\,\,\,\,\,\,\,\,+\, \eta\sum_{l=1}^{K}\int_{0}^{t}\|y_l\,\circ\,\lambda^n \,-\, y_l^n\|_s\,ds,\label{eq:A.22}
	\end{align}
	where the first inequality holds by \citet[Lemma 13.5.1]{Whitt_2002} and the fact that
	\begin{align*}
	&\psi\Big(\xi_j \,+ \sum_{l=1}^{K}q_{lj}\eta_j\!\int_{0}^{\cdot}\!y_l(s)\,ds\Big)\,\circ\, \lambda^n \\&\,\,\,\,\,\,=\, \psi\Big(\xi_j\,\circ\, \lambda^n \,+\sum_{l=1}^{K}q_{lj}\eta_j\!\int_{0}^{\lambda^n(\cdot)}\!\!\!y_l(s)\,ds\Big);
	\end{align*}
	see, e.g., \citet[Lemma 13.5.2]{Whitt_2002}. Combining these estimates, it then follows from (\ref{eq:A.18})--(\ref{eq:A.22}) that
	\begin{align}
	&\hspace{-1em}\max_{k\in [K]}\|y_k\,\circ\, \lambda^n \,-\, y^n_k\|_t\notag\\
	&\hspace{-1em}\,\,\,\,\,\,\le\, 2J\|(\xi,\zeta)\,\circ\, \lambda^n \,-\, (\xi^n,\zeta^n)\|_{T,J+K} \notag\\[2mm]&\,\,\,\,\,\,\,\,\,+\, (JK+1)\,\eta TM_T\|\dot{\lambda}^n\,-\,1\|_T\notag\\&\,\,\,\,\,\,\,\,\, +\,2\eta JK\!\int_{0}^{t}\!\!\max_{k\in[K]}\|y_k\,\circ\,\lambda^n \,-\, y_k^n\|_s\,ds,\label{eq:A.23}
	\end{align}
	for all $0\le t\le T$. Fixing $\epsilon>0$, it follows from (\ref{eq:A.17}) that there exists an $N\in\mathbb{N}$ such that for all $n\ge N$,
	\begin{align}
	2J\|(\xi,\zeta)\,\circ\, \lambda^n \,-\, (\xi^n,\zeta^n)\|_{T,J+K}&\,<\,\frac{\ep}{2e^{2\eta JKT}},\label{eq:A.24}\\
	(JK+1)\,\eta TM_T\|\dot{\lambda}^n\,-\,1\|_T&\,<\,\frac{\ep}{2e^{2\eta JKT}},\label{eq:A.25}\\
	\|\dot{\lambda}^n \,-\, 1\|_T&\,<\, \ep.\label{eq:A.29}
	\end{align}
	Then, by (\ref{eq:A.23})--(\ref{eq:A.25}), it follows that for all $n\ge N$,
	\begin{align}
	&\hspace{-1.5em}\max_{k\in [K]}\|y_k\,\circ\, \lambda^n \,-\, y^n_k\|_t\notag\\&<\, \frac{\ep}{e^{2\eta JKT}} \,+\, 2\eta JK\!\int_{0}^{t}\!\!\max_{k\in [K]}\|y_k\,\circ\,\lambda^n \,-\, y_k^n\|_s\,ds,\notag
	\end{align}
	for all $0\le t\le T$. By Gronwall's inequality (see, e.g., \citet[Lemma 4.1]{Pang_Talreja_Whitt_2007}) and the above displayed inequality, it follows that for all $n\ge N$,
	\begin{align}
	\max_{k\in [K]}\|y_k\,\circ\, \lambda^n \,-\, y^n_k\|_t\,<\, e^{2\eta JK(t-T)}\ep\,\le\,\ep,\label{eq:A.27}
	\end{align}
	for all $0\le t\le T$. Therefore, by (\ref{eq:A.29})--(\ref{eq:A.27}), we have that $\|y\,\circ\,\lambda^n -y^n\|_{T,K}\vee\|\dot{\lambda}^n - 1\|_T< \ep$ for all $n\ge N$. This completes the proof.

\section{Miscellaneous Proofs}\label{appendix:D}
Below is the proof of a result used in the proof of Lemma~\ref{Lemma:5.1} in Section~\ref{Section:5}.

\begin{lemma}\label{Lemma:C.1}
	If $\left\{X^n\right\}_{n=1}^{\infty}$ is a random sequence in $D$ such that $\|X^n\|_T\Rightarrow 0$ as $n\rightarrow\infty$ for all $T>0$, then $X^n\Rightarrow \mathbf{0}$ as $n\rightarrow\infty$.
\end{lemma}
\begin{proof}Note that $\|X^n\|_T\Rightarrow 0$ as $n\rightarrow\infty$ for all $T>0$ is equivalent to $\|X^n\|_T\overset{p}{\rightarrow} 0$ as $n\rightarrow\infty$ for all $T>0$, where the notation $\overset{p}{\rightarrow}$ is shorthand for ``converges in probability.'' We next show that $X^n\overset{p}{\rightarrow}\mathbf{0}$ as $n\rightarrow\infty$. This amounts to showing that for all $0<\ep<1$,
	\begin{align*}
	\lim_{n\rightarrow\infty}P\Big(\int_{0}^{\infty}\!\!e^{-t}\left[d_t\left(X^n, 0\right)\wedge 1\right]\,dt\,>\,\ep\Big)\,=\,0;
	\end{align*}
	see, e.g., \citet[Chapter 3, Section 3]{Whitt_2002}.
	But observe that 
	\begin{align}
	&\!\!\!\!\!\!P\Big(\int_{0}^{\infty}\!\!e^{-t}\big[\inf_{\lambda\in \Lambda_t}\left\{\|X^n\,\circ\,\lambda\|_t\vee \|\lambda -e\|_t\right\}\wedge 1\big]\,dt\,>\,\ep\Big)\notag\\
	&\le\, P\Big(\int_{0}^{\infty}\!\!e^{-t}\left[\|X^n\|_t\wedge 1\right]\,dt\,>\,\ep\Big)\notag\\
	&\le\, P\Big(\int_{0}^{T}\!\!e^{-t}\|X^n\|_t\,dt \,+ \int_{T}^{\infty}\!\!e^{-t}\,dt\,>\,\ep\Big),\label{eq5.4}
	\end{align}
	for all $T>0$. Now fix $T>0$ to be such that $\int_{T}^{\infty}e^{-t}\,dt=\ep/2$. It then follows from (\ref{eq5.4}) that
	\begin{align}
	&P\Big(\int_{0}^{T}\!\!e^{-t}\|X^n\|_t\,dt \,+ \int_{T}^{\infty}\!\!e^{-t}\,dt\,>\,\ep\Big)\nonumber\\&\quad=\, P\Big(\int_{0}^{T}\!\!e^{-t}\|X^n\|_t\,dt\,>\,\frac{\ep}{2}\Big)\nonumber\\&\quad\le\, P\Big(\|X^n\|_T\,>\,\frac{\ep}{2}\big(1-\frac{\ep}{2}\big)^{-1}\Big).\label{eq:blah}
	\end{align}
	Since $\|X^n\|_T\overset{p}{\rightarrow}0$ as $n\rightarrow\infty$, it follows from (\ref{eq:blah}) that $X^n\overset{p}{\rightarrow}\mathbf{0}$ as $n\rightarrow\infty$. Since convergence in probability implies convergence in distribution, it follows that $X^n\Rightarrow \mathbf{0}$ as $n\rightarrow\infty$, which completes the proof.\end{proof}

\end{document}